\documentclass[reqno,11pt]{amsart}

\usepackage{palatino}

\usepackage{amsmath}
\usepackage{amsfonts,amssymb,amsmath,latexsym,wasysym,mathrsfs,bbm,stmaryrd}
\usepackage[all]{xy}
\usepackage[usenames]{color}
\usepackage{epsfig}

\usepackage[mathcal]{euscript}

\usepackage{graphics,moreverb,hangcaption,gastex}

\def\R{\mathbb R}
\def\C{\mathbb C}

\def\N{\mathbb N}

\def\1{\mathbbm 1}

\def\e{\epsilon}
\def\d{\delta}
\def\s{\sigma}

\def\kk{\kappa}
\def\ll{\lambda}
\def\ex{\varepsilon}

\def\supp{\mathrm{supp}\,}

\newcommand{\interior}[1]{\raise0.2ex\hbox{$\displaystyle{\mathop{#1}^{\circ}}$}}
\newcommand{\mx}[1]{\mathbf{#1}}

\renewcommand\phi{\varphi}
\renewcommand\emptyset{\varnothing}
\def\t{\tau}

\newtheorem{theorem}{Theorem}[section]
\newtheorem*{theorem*}{Theorem}
\newtheorem*{theoremRc}{Theorem \ref{thm circular R}}
\newtheorem*{theoremNegmom}{Theorem \ref{thm neg moments}}
\newtheorem*{theoremmomentpoly}{Theorem \ref{thm moment polynomial}} 

\newtheorem{proposition}[theorem]{Proposition}
\newtheorem{definition}[theorem]{Definition}
\newtheorem{corollary}[theorem]{Corollary}

\newtheorem{lemma}[theorem]{Lemma}
\newtheorem{notation}[theorem]{Notation}
\numberwithin{equation}{section} 

\theoremstyle{remark}
\newtheorem*{remark*}{Remark}
\newtheorem{remark}[theorem]{Remark}

\theoremstyle{remark}

\long\def\symbolfootnote[#1]#2{\begingroup%
\def\thefootnote{\fnsymbol{footnote}}\footnote[#1]{#2}\endgroup}

\begin{document}

\title{Resolvents of $\mathscr{R}$-Diagonal Operators}
\author{Uffe Haagerup$^{(1)}$}
\address{$(1)$ Department of Mathematics and Computer Science, University of Southern Denmark \\ Campusvej 55 DK-5230 Odense M \; Denmark}
\email{haagerup@imada.sdu.dk}
\author{Todd Kemp$^{(2)}$}
\address{$(1)$ Department of Mathematics, MIT \\ 77 Massachusetts Avenue, Cambridge, MA \; 02139}
\email{tkemp@math.mit.edu}
\author{Roland Speicher$^{(3)}$}
\address{$(2)$ Department of Mathematics and Statistics, Queen's University \\ Jeffery Hall, Kingston, ON, Canada  \;  K7L\,3N6}
\email{speicher@mast.queensu.ca}

\begin{abstract} We consider the resolvent $(\ll-a)^{-1}$ of any $\mathscr{R}$-diagonal operator $a$ in a $\mathrm{II}_1$-factor.  Our main theorem (Theorem \ref{main theorem}) gives a universal asymptotic formula for the norm of such a resolvent.  En route to its proof, we calculate the $\mathscr{R}$-transform of the operator $|\ll-c|^2$ where $c$ is Voiculescu's circular operator, and give an asymptotic formula for the negative moments of $|\ll-a|^2$ for any $\mathscr{R}$-diagonal $a$.  We use a mixture of complex analytic and combinatorial techniques, each giving finer information where the other can give only coarse detail.  In particular, we introduce {\em partition structure diagrams} in Section \ref{sect Neg Moments}, a new combinatorial structure arising in free probability. \end{abstract}

\maketitle

\symbolfootnote[0]{(2) This work was partially supported by NSF Grant DMS-0701162.}

\vspace{-0.4in}

\section{Introduction} \label{sect intro}

\subsection{Motivation and Main Results} In this paper, we develop a number of universal norm estimates related to free probability theory.  We are, in particular, concerned with {\em $\mathscr{R}$-diagonal operators}, which are precisely defined on page \pageref{R-diagonal page}.  Originally introduced by Nica and Speicher in \cite{Nica Speicher Fields Paper}, they have been considered by many authors in papers including \cite{Haagerup 2,Haagerup Larsen,Haagerup Schultz,Haagerup Schultz 2,Kemp 2,Kemp Speicher,NSS,Larsen,Nica Speicher Duke Paper,Sniady Speicher}.  The class of $\mathscr{R}$-diagonal operators includes both Voiculescu's circular operator and Haar unitary operators, and is very large (the distribution of the real part of an $\mathscr{R}$-diagonal operator can be prescribed arbitrarily).  They are important in recent work on the invariant subspace conjecture relative to a $\mathrm{II}_1$-factor (cf.\ \cite{Haagerup Schultz,Haagerup Schultz 2}, and have been shown to maximize free entropy given distribution constraints.

\medskip

This paper is, in a sense, a continuation of \cite{Kemp Speicher} and \cite{Kemp 2}, which examined an important norm inequality (the Haagerup inequality, \cite{Haagerup 1}), originally in the context of Haar unitary operators, generalized to all $\mathscr{R}$-diagonal elements.  The Haagerup inequality compares operator norm to $L^2$-norm for homogeneous (non-commutative) polynomials in operators.  In \cite{Kemp 2}, the second author considered an alternate formulation based on the dilation $a\mapsto ra$ for $r\in(0,1)$, acting on the $C^\ast$-algebra generated by a family of free generators $a$.  Of interest are the elements $1+ra+(ra)^2+\cdots = (1-ra)^{-1}$. In that context, he proved a non-sharp version of the lower-bound in Theorem \ref{main theorem} below, for a sub-class of $\mathscr{R}$-diagonal operators (those with non-negative free cumulants).  The current paper can be viewed as providing the {\em sharp} norm inequality, for {\em all} $\mathscr{R}$-diagonal operators.

\medskip

Our main theorem, Theorem \ref{main theorem}, gives the precise rate of norm blow-up of the resolvent of an $\mathscr{R}$-diagonal operator near its spectral radius.  It is worthy of note for two reasons.  First, it is notoriously difficult to calculate resolvent blow-up rates, while we have calculated the rate for {\em all} $\mathscr{R}$-diagonal resolvents.  Second, the result is {\em universal}: the rate is always polynomial with exponent $-3/2$, and the constant is a product of a uniform factor with a quantity determined only by the $4th$ moment of the operator $a$.

\begin{theorem} \label{main theorem} Let $a$ be an $\mathscr{R}$-diagonal operator in a $\mathrm{II}_1$ factor $\mathscr{A}$ with trace $\phi$, normalized so that $\phi(aa^\ast) = \|a\|_2^2 = 1$.  Set $v(a) = \|a\|_4^4-1$.  Then $v(a)>0$ iff $a$ is not a Haar unitary, and in this case, for $\ll>1$,
\begin{equation} \label{eq main theorem}
\| (\ll-a)^{-1} \| \sim \sqrt\frac{27}{32}\,\sqrt{v(a)}\,\frac{1}{(\ll-1)^{3/2}} \quad \text{as} \quad \ll\downarrow 1.
\end{equation}
\end{theorem}

In proving Theorem \ref{main theorem}, we develop several auxiliary results of independent interest.  The special case that $a$ is Voiculescu's circular element $c$ affords an example where non-asymptotic calculations may be done completely explicitly.  In that case, we prove the following result on page \pageref{theorem circular R page}, which (as we show) can be used to prove this special case of the main theorem.

\begin{theoremRc}  Let $c$ be a circular operator of unit variance, and let $\ll\in\R$.  Then
\[ \mathscr{R}_{|\ll-c|^2}(z) = \frac{1}{1-z} + \frac{\ll^2}{(1-z)^2}. \]
\end{theoremRc}

Theorem \ref{thm circular R} is proved following its statement via combinatorial techniques; we reprove it in Section \ref{sect analytic approach to R} using the analytic techniques developed there.  We go on to use that analysis to calculate, to leading order, the negative moments of the operator $|\ll-a|^2$ for any $\mathscr{R}$-diagonal $a$.  The result, which appears on page \pageref{theorem neg moments page}, follows.

\begin{theoremNegmom} Let $k$ be a non-negative integer, and let $a$ satisfy the conditions of Theorem \ref{main theorem}. Then as $\ll\downarrow 1$,
\[ \phi(|\ll-a|^{-2(k+1)}) \sim C^{(2)}_k\frac{v(a)^k}{(\ll^2-1)^{3k+1}}, \]
where $C^{(2)}_k = \frac{1}{2k+1}\binom{3k}{k}$ are the (type $2$) Fuss-Catalan numbers.
\end{theoremNegmom}
%As $\ll\downarrow 1$, $\ll^4-1\downarrow 0$, and so Theorem \ref{thm neg moments} could be stated  with constant $v(a)^k$ in the numerator rather than the above expression.  Indeed, that is how it appears on page  \pageref{theorem neg moments page}, and this slightly less precise asymptotic statement is the extent of what can be proved using our analytic techniques.
We give two different proofs of this theorem: one complex analytic, in Section \ref{sect general case}, and the other combinatorial, in Section \ref{sect PSD 1}.  Theorem \ref{thm neg moments} by itself yields the sharp lower bound of Theorem \ref{main theorem}, as detailed in Section \ref{sect PSD 2}; the analytic techniques of Section \ref{sect general case} extend to prove this bound is also sharp from above.  In addition, our combinatorial approach demonstrates that the negative moments are in fact polynomials in appropriate quantities.  The theorem, appearing on page \pageref{thm moment polynomial page}, is as follows.

\begin{theoremmomentpoly}  Let $k\ge 0$, and let $a$ satisfy the conditions of Theorem \ref{main theorem}.  Then there is a {\em polynomial} $P_{k+1}^a$ in two variables so that
\[ \phi(|\ll-a|^{-2(k+1)}) = P_{k+1}^a\left(\frac{1}{\ll^2-1},\frac{1}{\ll^2}\right). \]
for $\ll>1$.
\end{theoremmomentpoly}

The proof of Theorem \ref{thm moment polynomial} led to the development of a new class of combinatorial objects we call {\em partition structure diagrams}, introduced in Section \ref{sect PSD 1}.

\subsection{Background}

Following is a brief description of those results and techniques from both the complex analytic and combinatorial sides of free probability theory that we use in this paper.  They are here largely as a means to fix notation.  The reader is directed to the papers \cite{Haagerup Larsen,Haagerup Schultz,Kemp 2,Kemp Speicher} and the book \cite{Nica Speicher Book} for further reading.

\medskip

The arena for all of what follows is a $\mathrm{II}_1$-factor $\mathscr{A}$ with trace $\phi$.  Operators in $\mathscr{A}$ are {\em non-commutative random variables}.  If $x\in\mathscr{A}$ is self-adjoint, it has a spectral resolution $E^x$ whose projections are in $\mathscr{A}$; the measure $\mu_x = \phi\circ E^x$ is the {\em distribution}  or {\em spectral measure} of $x$.  Equivalently, $\mu_x$ is the unique probability measure on $\R$ whose moments $\int t^n \,\mu_x(dt)$ are given by the moments $\phi(x^n)$ for $n\in\N$.  Even if $x\in\mathscr{A}$ is not self-adjoint, we therefore refer to the collection of all $\phi$-moments of monomials in $x$ and $x^\ast$ as the {\em distribution} of $x$.

\medskip

Given a probability measure $\mu$ on $\R$, its {\em Cauchy transform} $G_\mu$ is the analytic function defined in the upper half-plane $\C_+$ by

\begin{equation} \label{eq Cauchy transform} G_\mu(z) = \int_{\R} \frac{1}{z-t}\,\mu(dt), \quad z\in\C_+. \end{equation}
The {\em $\mathscr{R}$-transform} of the measure, $\mathscr{R}_\mu$, is the analytic function defined in a neighbourhood of $0$ determined by the functional equation

\begin{equation} \label{eq G-R functional eqn} G_\mu(\mathscr{R}_\mu(z) + 1/z) = z, \quad z \in \C_+, \quad |z|\text{ small}.\end{equation}

\medskip

\noindent For a known $\mathscr{R}$-transform $\mathscr{R}_\mu$, Equation \ref{eq G-R functional eqn} in fact determines $G_\mu$ on a sector in $\C_+$, and thence on all of $\C_+$ by analytic continuation, modulo the asymptotic restriction that $\lim_{|z|\to\infty} zG_\mu(z) = 1$.  This relationship shows that the measure $\mu$ can be recovered from its $\mathscr{R}$-transform, via the {\em Stieltjes inversion formula}:

\begin{equation} \label{eq Stieltjes inversion} \mu(dt) =  -\frac{1}{\pi}\lim_{\e\downarrow 0}\Im G_\mu(t+i\e)\,dt. \end{equation}

\noindent Equation \ref{eq Stieltjes inversion} should be interpreted in weak form (that $\int f(t)\,\mu(dt) = -\frac{1}{\pi} \lim_{\e\downarrow 0} \int f(t) \Im G_\mu(t+i\e)\,dt$ for $f\in C_c(\R)$) in general, but in the case that $\mu$ has a density $\rho$ with respect to Lebesgue measure $\mu(dt) = \rho(t)\,dt$, Equation \ref{eq Stieltjes inversion} yields $\rho(t)$ as the limit on the right-hand-side.

\medskip

Given $x,y\in\mathscr{A}$ self-adjoint, they are called {\em free} if $\mathscr{R}_{\mu_{x+y}} = \mathscr{R}_{\mu_x}+\mathscr{R}_{\mu_y}$.  Freeness can be written in other more combinatorial forms by considering the additivity of $\mathscr{R}$-transforms as a collection of statements about Taylor coefficients.  It is easy to verify that $\mathscr{R}_\mu(0) = 0$ for any measure $\mu$, and so in general we have
\[ \mathscr{R}_\mu(z) = \kk_1(\mu) + \kk_2(\mu)\,z + \kk_3(\mu)\,z^2 + \cdots \]
for some scalars $\kk_n(\mu)$ called the {\em free cumulants} of $\mu$.  Thinking of $\mathscr{R}$ and $\kk_n$ indexed by a random variable rather than its distribution, we can polarize and express $\kk_n$ as an $n$-linear functional: $\kk_n[x_1,\ldots,x_n]$, where $\kk_n[x,\ldots,x] = \kk_n(\mu_x)$.  In this language, freeness can be stated thus: random variables are free if all their mixed free cumulants vanish.  This parallels the classical connection between independence of random variables and their classical cumulants (also known as semi-invariants).  It also provides an extension of the notion of freeness to any collection of (not necessarily self-adjoint) random variables.

\medskip

The relationship between moments and free cumulants is given by the {\em moment cumulant formula}:
\begin{equation} \label{eq moment cumulant} \phi(x_1\cdots x_n) = \sum_{\pi\in NC(n)} \kk_{\pi}[x_1,\ldots,x_n].
\end{equation}
Here $NC(n)$ denotes the lattice of {\em non-crossing partitions} of the ordered set $\{1,\ldots,n\}$.  Given a partition $\pi = \{B_1,\ldots,B_r\}$ (with $B_j\subseteq\{1,\ldots,n\}$), the quantity $\kk_{\pi}[x_1,\ldots,x_n]$ is equal to the product of the $r$ terms $\kk(B_j)[x_1,\ldots,x_n]$, where if $B = \{i_1,\ldots,i_m\}$ then $\kk(B)[x_1,\ldots,x_n] = \kk_m[x_{i_1},\ldots,x_{i_m}]$.  For example, if $\pi = \left\{\{1,4,5\},\{2,3\}\right\}$, then $\kk_\pi[x_1,\ldots,x_5] = \kk_3[x_1,x_4,x_5]\,\kk_2[x_2,x_3]$.

\medskip

\label{R-diagonal page}
An non-commutative random variable $a\in\mathscr{A}$ is called {\bf $\mathscr{R}$-diagonal} if, among all mixed free cumulants in $a,a^\ast$, the only non-zero ones are among
\[ \kk_{2n}[a,a^\ast,\ldots,a,a^\ast],\quad \kk_{2n}[a^\ast,a,\ldots,a^\ast,a], \]
for some positive integer $n$.  Prominent examples of $\mathscr{R}$-diagonal operators are Haar unitary operators and Voiculescu's circular operator $c$, often represented in the form $c = \frac{1}{\sqrt{2}}(s + i s')$ where $s,s'$ are free semicircular random variables --- self-adjoint operators with distribution $\mu_{s}(dt) = \mu_{s'}(dt) = \frac{1}{2\pi}\sqrt{4-t^2}\1_{\{|t|\le 2\}}$).  The class of $\mathscr{R}$-diagonal operators is closed under free sum and product and taking powers, and for any compactly-supported probability measure $\mu$ on $\R_+$ there is an $\mathscr{R}$-diagonal operator $a$ with $\mu_{aa^\ast} = \mu$.

\medskip

The general moment--cumulant formula takes a special form in the case of $\mathscr{R}$-diagonal operators.  Let $a$ be $\mathscr{R}$-diagonal, and consider any monomial in $a,a^\ast$: $a^{\ast n_1} a^{m_1}\cdots a^{\ast n_k} a^{m_k}$ for $n_1,\ldots,m_k$ non-negative integers.  The following formula is a consequence of the definition of $\mathscr{R}$-diagonality, and is proved in \cite{Kemp 2}.

\begin{equation} \label{eq R-diag moment cumulant}
\phi(a^{\ast n_1} a^{m_1}\cdots a^{\ast n_k} a^{m_k}) = \sum_{\pi\in NC(n_1,m_1,\ldots,n_k,m_k)} \kk_{\pi}[a^{,n_1},a^{\ast,m_1},\ldots,a^{,n_k},a^{\ast,m_k}].
\end{equation}
Here $NC(n_1,m_2,\ldots,n_k,m_k)$ denotes the set of non-crossing partitions $\pi$ of the list of length $n_1+m_1+\cdots+n_k+m_k$ with the property that each block of $\pi$ alternately connects $a$s and $a^\ast$s.  Commas have been added in the exponents of the cumulants to emphasize that the arguments are {\em not products}.  The set of pairings with this property is denoted $NC_2(n_1,m_1,\ldots,n_k,m_k)$.

\begin{figure}[htbp]
\begin{center}
\input{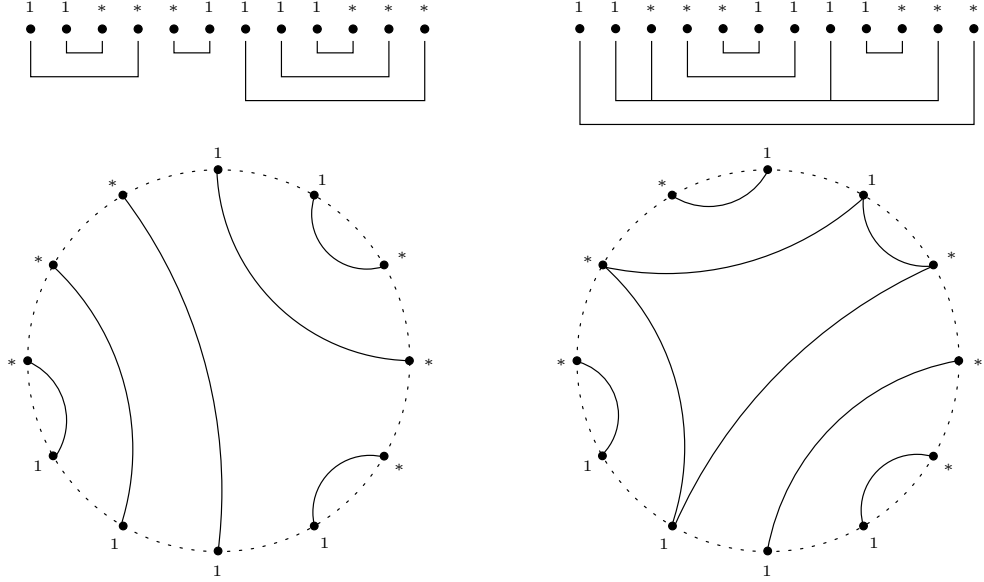}
\caption{\label{fig R-diag partitions} Two non-crossing partitions in $NC(2,3,4,3)$, represented in linear format (top) and on the disc (bottom); the latter representation will be more useful in Section \ref{sect PSD 1}.  The first partition is a pairing in $NC_2(2,3,4,3)$.}
\end{center}
\end{figure}

In fact, it {\em is} sometimes useful to consider cumulants with products as arguments (thus the need for the commas above to distinguish).  The following theorem (Theorem 11.12 in \cite{Nica Speicher Book}) is a powerful computational tool we will use in Section \ref{sect circular}.

\begin{theorem} \label{thm prod cumulant} Let $\mx{i}=(i_1,\ldots,i_n)$ be an $n$-tuple of natural numbers, and let $a_1,\ldots, a_{|\mx{i}|}$ be random variables in a non-commutative probability space.  Consider the products
\[ A_1 = a_1\cdots a_{i_1} \quad A_2 = a_{i_1+1}\cdots a_{i_1+i_2} \quad \ldots \quad A_n = a_{i_1+\cdots+i_{n-1}+1}\cdots a_{i_1+\cdots+i_n}. \]
The free cumulants of these product variables are given in terms of the free cumulants of the $a_j$ themselves by
\begin{equation} \label{eq prod cumulant}
\kk_n[A_1,\ldots,A_n] = \sum_{\pi\in NC(|\mx{i}|)\atop \pi\vee \widehat{0_n} = 1_{|\mx{i}|}} \kk_\pi[a_1,\ldots,a_{|\mx{i}|}],
\end{equation}
where $\widehat{0_n}$ is the partition whose blocks are the intervals $\{1,\ldots,i_1\}, \{i_1+1,\ldots,i_1+i_2\}, \ldots, \{i_1+\cdots +i_{n-1}+1,\ldots,i_1+\cdots+i_n\}$.
\end{theorem}

\begin{remark} \label{rk expand list} For notational convenience, we will express the relationship between the tuples $\mx{A} = [A_1,\ldots,A_n]$ and $\mx{a} = [a_1,\ldots,a_{\mx{i}}]$ in Equation \ref{eq prod cumulant} by $\mx{a} = \widehat{\mx{A}}$. For example, if $\mx{i}=(2,1,2,1)$ then $\mx{A} = [a_1a_2, a_3, a_4a_5, a_6]$ and so $\widehat{\mx{A}} = [a_1,a_2,a_3,a_4,a_5,a_6]$.  The use of the notation is in summations over $\mx{i}$, where it is cumbersome to explicitly enumerate the break-points between products. \end{remark}

\noindent The $\vee$ in Equation \ref{eq prod cumulant} denotes the join in the lattice $NC(|\mx{i}|)$. The meaning of the condition $\pi\vee\widehat{0_n} = 1_{|\mx{i}|}$ is as follows: the blocks of $\pi$ must connect the blocks of $\widehat{0_n}$.  To be precise: given any two points $p,q$ in $\{1,\ldots,|\mx{i}|\}$, there must be a path $p=p_1\sim_{\s_1} p_2 \sim_{\s_2} \cdots \sim_{\s_{r-1}} p_r = q$ where $\s_j \in\{\pi,\widehat{0_n}\}$ for $j=1\ldots r-1$.  Indeed, the sequence $\s_j$ can be chosen to alternate between $\pi$ and $\widehat{0_n}$.  Figure \ref{fig connect} gives examples.

\begin{figure}[htbp]
\begin{center}
\unitlength=4.5pt
    \begin{picture}(95, 15)(0,-12)
    \thinlines
    \gasset{Nw=1, Nh=1, Nmr=1,Nframe=n,Nfill=y, Nmarks=n, ExtNL=y, NLdist=2,AHnb=0}
    \node(s1)(0,0){}
    \node(s2)(5,0){}
    \drawrect[dash={0.5 1}0,Nfill=n,Nframe=y](-2,-2,7,2)
    \node(s3)(10,0){}
    \drawrect[dash={0.5 1}0,Nfill=n,Nframe=y](8,-2,12,2)
    \node(s4)(15,0){}
    \node(s5)(20,0){}
    \drawrect[dash={0.5 1}0,Nfill=n,Nframe=y](13,-2,22,2)
    \node(s6)(25,0){}
    \drawrect[dash={0.5 1}0,Nfill=n,Nframe=y](23,-2,27,2)
    
    \drawline[linewidth=0.2](0,-2.8)(0,-8)(20,-8)(20,-2.8)(20,-8)(25,-8)(25,-2.8)
    \drawline[linewidth=0.2](5,-2.8)(5,-6)(10,-6)(10,-2.8)(10,-6)(15,-6)(15,-2.8)
    
     \node[Nfill=n](s7)(12.5,-14){$\pi_1$}
     
    \node(t1)(35,0){}
    \node(t2)(40,0){}
    \drawrect[dash={0.5 1}0,Nfill=n,Nframe=y](33,-2,42,2)
    \node(t3)(45,0){}
    \drawrect[dash={0.5 1}0,Nfill=n,Nframe=y](43,-2,47,2)
    \node(t4)(50,0){}
    \node(t5)(55,0){}
    \drawrect[dash={0.5 1}0,Nfill=n,Nframe=y](48,-2,57,2)
    \node(t6)(60,0){}
    \drawrect[dash={0.5 1}0,Nfill=n,Nframe=y](58,-2,62,2)
    
    \drawline[linewidth=0.2](35,-2.8)(35,-8)(60,-8)(60,-2.8)
    \drawline[linewidth=0.2](40,-2.8)(40,-6)(45,-6)(45,-2.8)(45,-6)(50,-6)(50,-2.8)(50,-6)(55,-6)(55,-2.8)
    
     \node[Nfill=n](t7)(47.5,-14){$\pi_2$}
     
    \node(u1)(70,0){}
    \node(u2)(75,0){}
    \drawrect[dash={0.5 1}0,Nfill=n,Nframe=y](68,-2,77,2)
    \node(u3)(80,0){}
    \drawrect[dash={0.5 1}0,Nfill=n,Nframe=y](78,-2,82,2)
    \node(u4)(85,0){}
    \node(u5)(90,0){}
    \drawrect[dash={0.5 1}0,Nfill=n,Nframe=y](83,-2,92,2)
    \node(u6)(95,0){}
    \drawrect[dash={0.5 1}0,Nfill=n,Nframe=y](93,-2,97,2)
    
    \drawline[linewidth=0.2](70,-2.8)(70,-6)(75,-6)(75,-2.8)
    \drawline[linewidth=0.2](80,-2.8)(80,-6)(85,-6)(85,-2.8)    
    \drawline[linewidth=0.2](90,-2.8)(90,-6)(95,-6)(95,-2.8)
    
     \node[Nfill=n](u7)(82.5,-14){$\pi_3$}

    \end{picture}
\caption{\label{fig connect}Three examples of partitions in $NC(6)$.  With multi-index $\mx{i}=(2,1,2,1)$, $\pi_1$ and $\pi_2$ do connect all the blocks of $\widehat{0_4}$, and so would be included in the sum in Equation \ref{eq prod cumulant}; $\pi_3$ leaves the first block of $\widehat{0_4}$ isolated, and so is not included in this sum.}
\end{center}
\end{figure}
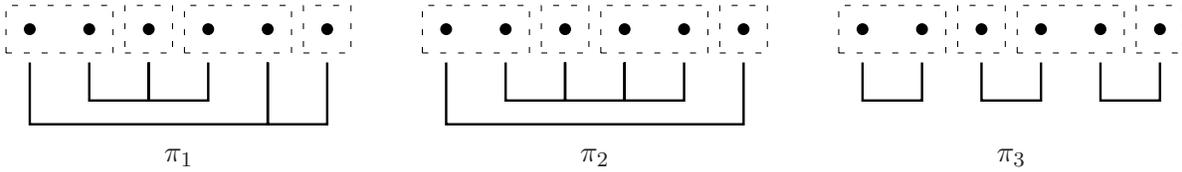

\subsection{Organization}  This paper is organized as follows.  In Section \ref{sect circular}, we address our main theorem through the special case of a circular operator $c$, the pre-eminent example of an $\mathscr{R}$-diagonal element.  In Section \ref{sect combinatorial R}, we calculate the $\mathscr{R}$-transform of the operator $|\ll-c|^2$ for any scalar $\ll$, using combinatorial means (primarily judicious application of Theorem \ref{thm prod cumulant}).  In Section \ref{sect circular support}, we us this $\mathscr{R}$-transform to explicitly determine the support of the spectral measure of $|\ll-c|^2$; its left boundary point represents the norm of the resolvent $(\ll-c)^{-1}$.  In this case, the measure itself can be completely determined.

\medskip

The exact calculations of Section \ref{sect circular support} cannot be extended to the general $\mathscr{R}$-diagonal case, and so we proceed to develop analytic arguments to prove the asymptotic statement of Theorem \ref{main theorem}.  In Section \ref{sect general case}, we demonstrate the power of working with the symmetrizations of spectral measures (via Equations \ref{eq symmetric sum} and \ref{eq R transform sum}).  Section \ref{sect analytic approach to R} shows how Theorem \ref{thm circular R} can be obtained directly from these analytic means.  Section \ref{sect analytic continuation} gives a general analytic continuation argument from the above-mentioned equations that yields a useful power-series inversion formula, which is then used in Section \ref{sect neg moments} to calculate (to leading order) the negative absolute moments of the resolvent $(\ll-a)^{-1}$ (theorem \ref{thm neg moments}).  These techniques are then pushed through to give a complete proof of Theorem \ref{main theorem} in Section \ref{sect proof of main theorem}.

\medskip

Finally, in Section \ref{sect Neg Moments}, we examine the combinatorial structures underlying the negative absolute moments of the resolvent $(\ll-a)^{-1}$.  In Section \ref{sect PSD 1}, we introduce {\em partition structure diagrams}, a new way to view the basic structure of partitions appearing in Equation \ref{eq R-diag moment cumulant} for $\mathscr{R}$-diagonal operators, and use their enumeration to provide a bijective combinatorial proof of the refinement (Theorem \ref{thm moment polynomial}) of Theorem \ref{thm neg moments}.  Then, in Section \ref{sect PSD 2}, we show how knowledge of the asymptotics of negative moments alone can be used to recapture the sharp lower-bound of Theorem \ref{main theorem}.

\section{The Circular Case} \label{sect circular}

Let $c$ be a circular operator of unit variance in a $\mathrm{II}_1$-factor $\mathscr{C}$.  Since $c$ is $\mathscr{R}$-diagonal, its spectral radius is $\|c\|_2 =1$ by our choice of variance.  Hence the resolvent
\[ R_c(\lambda) = (\lambda-c)^{-1} \]
is a $\mathscr{C}$-valued analytic function on the domain $|\lambda|>1$.  Our goal in this section is to calculate $\|R_c(\lambda)\|$ to leading order as $|\lambda|\downarrow 1$.

\begin{remark} \label{rk c rot inv}  Note that $c$ is rotationally-invariant; it follows that if $\theta\in\R$ then $\|R_c(\lambda e^{-i\theta})\| = \|R_c(\lambda)\|$.  Hence, we restrict our attention to the case $\lambda>1$ in $\R$. \end{remark}

For $\ll>1$, define the positive operator $T_\ll$ by 
\begin{equation} \label{eq T} T=T_\ll = R_c(\ll)^\ast R_c(\ll) = (\ll-c^{\ast})^{-1} (\ll-c)^{-1}. \end{equation}

\noindent Note that $\|R_c(\ll)\|^2 = \|T_\ll\|$.  What's more, since $T>0$ it follows that $\|T\| =  \inf\mathrm{spec}\, (T^{-1})$, and $T^{-1} = (\ll-c)(\ll-c^\ast) = |\ll-c|^2$ is an operator we can handle with combinatorial techniques.  In particular, we will now calculate the $\mathscr{R}$-transform of this operator, which will allow us to calculate the spectral measure of $T^{-1}$ through Equations \ref{eq Cauchy transform} and \ref{eq G-R functional eqn}.

\subsection{The $\mathscr{R}$-transform of $|\ll-c|^2$} \label{sect combinatorial R}

\label{theorem circular R page}

\begin{theorem} \label{thm circular R} Let $c$ be a circular operator of unit variance, and let $\ll\in\R$.  Then
\begin{equation} \label{eq circular R} \mathscr{R}_{|\ll-c|^2}(z) = \frac{1}{1-z} + \frac{\ll^2}{(1-z)^2}. \end{equation}
\end{theorem}

\begin{remark} \label{rk circular R} We find the formula in Equation \ref{eq circular R} interesting in its own right.  It mirrors a similar formula for the semicircular equivalent provided in \cite{Yoshida}; if $s$ is a semicircular operator of variance $1$,
\[ \mathscr{R}_{(\ll-s)^2}(z) = \frac{1}{1-z} + \frac{\ll^2}{(1-2z)^2}. \]
Their techniques are entirely analytic: indeed, one can calculate the Cauchy-transform of $(\ll-s)^2$ from that of $\ll-s$, the latter of which is well-known, and then the $\mathscr{R}$-transform is achieved through Equation \ref{eq G-R functional eqn}.  Our approach below is markedly different, using only combinatorial techniques; however, analytic techniques will be developed to study the more general case in later sections, and we will rederive Equation \ref{eq circular R} using those techniques in Section \ref{sect analytic approach to R}
\end{remark}

\begin{proof} Expand $|\ll-c|^2 = \ll^2 -\ll(c+c^\ast) + cc^\ast$.  Denote $\alpha_1 = -\ll(c+c^\ast)$ and $\alpha_2 = cc^\ast$, so that $|\ll-c|^2 = \ll^2 + \alpha_1+\alpha_2$.  The constant $\ll^2$ is free from any operator, and so we have an initial simplification
\begin{equation} \label{eq c R-transform 1} \mathscr{R}_{|\ll-c|^2}(z) = \mathscr{R}_{\ll^2}(z) + \mathscr{R}_{\alpha_1+\alpha_2}(z) = \ll^2 + \mathscr{R}_{\alpha_1+\alpha_2}(z). \end{equation}
Now, $\alpha_1,\alpha_2$ are certainly not free.  We calculate the $\mathscr{R}$-transform of their sum as a power-series whose coefficients are free cumulants:
\begin{equation} \label{eq c R-transform 2} \mathscr{R}_{\alpha_1+\alpha_2}(z) = \sum_{n\ge 1} \kk_n[\alpha_1+\alpha_2,\ldots,\alpha_1+\alpha_2]\, z^{n-1}. \end{equation}
The free cumulant $\kk_n$ is a multilinear function, and so we can expand
\begin{equation} \label{eq expand k} 
\kk_n[\alpha_1+\alpha_2,\ldots,\alpha_1+\alpha_2] = \sum_{(i_1,\ldots,i_n)\in\{1,2\}^n} \kk_n[\alpha_{i_1},\ldots,\alpha_{i_n}].
\end{equation}
We will shortly see that the vast majority of the $2^n$ terms in the sum in Equation \ref{eq expand k} are $0$.  To ease notation, let $\mx{i}$ denote the multi-index $(i_1,\ldots,i_n)$, and denote the $n$-tuple $\alpha_{i_1},\ldots,\alpha_{i_n}$ as $\alpha_{\mx{i}}$.  Since $\alpha_2 = cc^\ast$ is a product, for each $\mx{i}$ we can expand the cumulant in Equation \ref{eq expand k} using Equation \ref{eq prod cumulant}.
\begin{equation} \label{eq sum 1} \kk_n[\alpha_\mx{i}] =\sum_{\pi\in NC(|\mx{i}|)\atop \pi\vee\widehat{0_n}=1_{|\mx{i}|}} \kk_\pi[\widehat{\alpha_\mx{i}}]. \end{equation}
As in Remark \ref{rk expand list}, the list $\widehat{\alpha_\mx{i}}$ is the expanded list of products from $\alpha_\mx{i}$.  For example, if $\mx{i}=(2,1,2,1)$ so that $\alpha_\mx{i} = [\alpha_2,\alpha_1,\alpha_2,\alpha_1] = [cc^\ast,\alpha_1,cc^\ast,\alpha_1]$, then $\widehat{\alpha_\mx{i}} = [c,c^\ast,\alpha_1,c,c^\ast,\alpha_1]$.

\medskip

Since $\alpha_1 = -\ll(c+c^\ast)$, any such cumulant $k_\pi[\widehat{\alpha_\mx{i}}]$ can be expanded into a sum of cumulants $k_\pi$ of a list of $c$s and $c^\ast$s.  Since the only non-vanishing block $\ast$-cumulants of $c$ are $\kk_2[c,c^\ast] = \kk_2[c^\ast,c] =1$, this means that {\em the only $\pi$ which can contribute to the sum \ref{eq sum 1} are non-crossing {\bf pairings}}.  This turns out to be an enormous simplification of the sum \ref{eq expand k}; the result is most of these terms are $0$.  Let us first consider the two endpoints.

\medskip

Suppose $\mx{i} = (1,1,\ldots,1)$.  The corresponding term in Equation \ref{eq expand k} is
\[ \kk_n[\alpha_1,\ldots,\alpha_1] = (-\ll)^n\kk_n[c+c^\ast,\ldots,c+c^\ast]. \]
Expanding this in $2^n$ terms, we have mixed cumulants in $c,c^\ast$, only two of which are non-vanishing: $\kk_2[c,c^\ast] = \kk_2[c^\ast,c] = 1$.  Hence,
\begin{equation} \label{eq endpoint 1} \kk_n[\alpha_1,\ldots,\alpha_1] = 2\ll^2 \1_{\{n=2\}}. \end{equation}

\medskip

On the other hand, suppose $\mx{i} = (2,2,\ldots,2)$.  The corresponding term in Equation \ref{eq expand k} is
\[ \kk_n[\alpha_2,\ldots,\alpha_2] = \kk_n[cc^\ast,cc^\ast,\ldots,cc^\ast]. \]
Employing Theorem \ref{thm prod cumulant} and the above observation that only pairings contribute, we can expand this cumulant as a sum,
\begin{equation} \label{eq sum 2}
\kk_n[cc^\ast,\ldots,cc^\ast] = \sum_{\pi\in NC_2(2n)\atop \pi\vee \widehat{0_n}=1_{2n}} \kk_\pi[c,c^\ast,\ldots,c,c^\ast].
\end{equation}
In this case, $\widehat{0_n} = \{\{1,2\},\{3,4\},\ldots,\{2n-1,2n\}\}$.  Let $\pi$ be any pairing that connects these blocks, and consider the block in $\pi$ containing $1$.  Since $\pi$ is non-crossing, the match to $1$ must be even (or there would be an odd number of points in between that could therefore not be paired in a non-crossing manner).   Suppose $1\sim_\pi {2k}$.  If $k<n$, then there can be no non-crossing path joining $2k+1$ to $2k$, since such a concatenation of pairings would have to cross the pairing $\{1,2k\}$, as demonstrated in Figure \ref{fig proof 1}.
\begin{figure}[htbp]
\begin{center}
\unitlength=5pt
    \begin{picture}(70, 15)(0,-10)
    \thinlines
    \gasset{Nw=1, Nh=1, Nmr=1,Nframe=n,Nfill=y, Nmarks=n, ExtNL=y, NLdist=2,AHnb=0}
    \node(s1)(0,0){$1$}
    \node(s2)(5,0){}
    \drawrect[dash={0.5 1}0,Nfill=n,Nframe=y](-2,-2,7,2)
    \node(s3)(10,0){}
    \node(s4)(15,0){}
    \drawrect[dash={0.5 1}0,Nfill=n,Nframe=y](8,-2,17,2)
    \node(s5)(30,0){}
    \drawedge[dash={0.2 1.5}0, sxo=4,exo=-4](s4,s5){}    
    \node(s6)(35,0){}
    \drawrect[dash={0.5 1}0,Nfill=n,Nframe=y](28,-2,37,2)   
    \node(s7)(40,0){}
    \node(s8)(45,0){$2k$}
    \drawrect[dash={0.5 1}0,Nfill=n,Nframe=y](38,-2,47,2)     
    \node(s9)(50,0){}
    \node(s10)(55,0){}
    \drawrect[dash={0.5 1}0,Nfill=n,Nframe=y](48,-2,57,2) 
    \node(s11)(60,0){}
    \node(s12)(65,0){}
    \drawrect[dash={0.5 1}0,Nfill=n,Nframe=y](58,-2,67,2) 

    \drawline[linewidth=0.2](0,-2.8)(0,-6)(45,-6)(45,-2.8){}
    \drawline[linewidth=0.2](22,-2.8)(22,-5.6){}
    \drawline[linewidth=0.2](22,-6.4)(22,-10)(50,-10)(50,-2.8){}

    \end{picture}
\caption{\label{fig proof 1}}
\end{center}
\end{figure}

\noindent Hence, it must be that $1\sim_\pi 2n$.  Now, consider the match to $2$: say $2\sim_\pi 2\ell+1$.  If $\ell>1$, then the point $3$ cannot be connected to $2$ with a path composed of blocks in $\pi$ and $\widehat{0_n}$: since $\pi$ is non-crossing, the match to $3$ must be either $4$ or lie within the blocks $\{5,6\},\ldots,\{2\ell-1,2\ell\}$.  None of these blocks can be connected to any other blocks of $\widehat{0_n}$ via $\pi$ without crossing the pairing $\{2,2\ell+1\}$.  Hence, it must be that $2\sim_\pi 3$.  This is demonstrated in Figure \ref{fig proof 2}.
\begin{figure}[htbp]
\begin{center}
\unitlength=5pt
    \begin{picture}(70, 15)(0,-12)
    \thinlines
    \gasset{Nw=1, Nh=1, Nmr=1,Nframe=n,Nfill=y, Nmarks=n, ExtNL=y, NLdist=2,AHnb=0}
    \node(s1)(0,0){$1$}
    \node(s2)(5,0){}
    \drawrect[dash={0.5 1}0,Nfill=n,Nframe=y](-2,-2,7,2)
    \node(s3)(10,0){}
    \node(s4)(15,0){}
    \drawrect[dash={0.5 1}0,Nfill=n,Nframe=y](8,-2,17,2)
    \node(s5)(20,0){}
    \node(s6)(25,0){}
    \drawrect[dash={0.5 1}0,Nfill=n,Nframe=y](18,-2,27,2)
    \drawrect[dash={0.2 0.2}0,Nfill=n,Nframe=y](7.5,-8,27.5,4)
    \node[Nfill=n](t)(17,-8){isolated}
    \node(s7)(30,0){$2\ell+1$}
    \node(s8)(35,0){}
    \drawrect[dash={0.5 1}0,Nfill=n,Nframe=y](28,-2,37,2)
    \node(s9)(50,0){}
    \drawedge[dash={0.2 1.5}0, sxo=4,exo=-4](s8,s9){}    
    \node(s10)(55,0){}
    \drawrect[dash={0.5 1}0,Nfill=n,Nframe=y](48,-2,57,2)   
    \node(s11)(60,0){}
    \node(s12)(65,0){}
    \drawrect[dash={0.5 1}0,Nfill=n,Nframe=y](58,-2,67,2)

    \drawline[linewidth=0.2](0,-2.8)(0,-12)(65,-12)(65,-2.8){}
    \drawline[linewidth=0.2](5,-2.8)(5,-10)(30,-10)(30,-2.8){}

    \end{picture}
\caption{\label{fig proof 2}}
\end{center}
\end{figure}

\medskip

Iterating this argument shows that, in fact, there is only {\em one} pairing $\pi\in NC_2(2n)$ for which $\pi\vee \widehat{0_n} = 1_{2n}$: the pairing $\varpi_n = \{\{1,2n\},\{2,3\},\{3,4\},\ldots,\{2n-2,2n-1\}\}$ pictured in Figure \ref{fig proof 3}.

\begin{figure}[htbp]
\begin{center}
\unitlength=5pt
    \begin{picture}(70, 10)(0,-8)
    \thinlines
    \gasset{Nw=1, Nh=1, Nmr=1,Nframe=n,Nfill=y, Nmarks=n, ExtNL=y, NLdist=2,AHnb=0}
    \node(s1)(0,0){}
    \node(s2)(5,0){}
    \drawrect[dash={0.5 1}0,Nfill=n,Nframe=y](-2,-2,7,2)
    \node(s3)(10,0){}
    \node(s4)(15,0){}
    \drawrect[dash={0.5 1}0,Nfill=n,Nframe=y](8,-2,17,2)
    \node(s5)(20,0){}
    \node(s6)(25,0){}
    \drawrect[dash={0.5 1}0,Nfill=n,Nframe=y](18,-2,27,2)
    \node(s7)(30,0){}
    \node(s8)(35,0){}
    \drawrect[dash={0.5 1}0,Nfill=n,Nframe=y](28,-2,37,2)
    \node(s9)(50,0){}
    \drawedge[dash={0.2 1.5}0, sxo=4,exo=-4](s8,s9){}    
    \node(s10)(55,0){}
    \drawrect[dash={0.5 1}0,Nfill=n,Nframe=y](48,-2,57,2)   
    \node(s11)(60,0){}
    \node(s12)(65,0){}
    \drawrect[dash={0.5 1}0,Nfill=n,Nframe=y](58,-2,67,2)
    
    \drawline[linewidth=0.2](5,-2.8)(5,-5)(10,-5)(10,-2.8){}
    \drawline[linewidth=0.2](15,-2.8)(15,-5)(20,-5)(20,-2.8){}
    \drawline[linewidth=0.2](25,-2.8)(25,-5)(30,-5)(30,-2.8){}
    \drawline[linewidth=0.2](35,-2.8)(35,-5)(40,-5)(40,-2.8){}        
    \drawline[linewidth=0.2](45,-2.8)(45,-5)(50,-5)(50,-2.8){}
    \drawline[linewidth=0.2](55,-2.8)(55,-5)(60,-5)(60,-2.8){}
    \drawline[linewidth=0.2](0,-2.8)(0,-8)(65,-8)(65,-2.8){}
                
    \end{picture}
\caption{\label{fig proof 3}The pairing $\varpi_n$ is the unique pairing for which $\varpi_n\vee \widehat{0_n}=1_{2n}$ where $\widehat{0_n} = \{\{1,2\},\{3,4\},\ldots,\{2n-1,2n\}\}$.}
\end{center}
\end{figure}
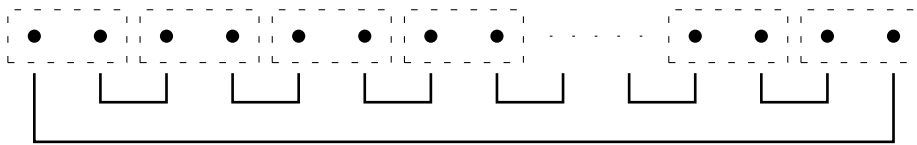

\noindent Hence, we have from Equation \ref{eq sum 2},
\begin{equation} \label{eq endpoint 2} \kk_n[\alpha_2,\ldots,\alpha_2] = \kk_{\varpi_n}[c,c^\ast,\ldots,c,c^\ast] = \kk_2[c,c^\ast]\kk_2[c^\ast,c]^{n-1} = 1. \end{equation}

\medskip

Now, we must consider the remaining $2^{n-1}$ terms in Equation \ref{eq expand k}, all incorporating some mixture of $\alpha_1$s and $\alpha_2$s.  The following lemmas show that only $n$ of these terms are non-zero.

\begin{lemma} \label{lem 2 1s} Suppose $\mx{i}$ is a length $n$ string containing both $1$s and $2$s, and let $\widehat{0_n}$ denote the corresponding interval partition.  If there exists $\pi\in NC_2(|\mx{i}|)$ such that $\pi\vee\widehat{0_n} = 1_{|\mx{i}|}$, then $\mx{i}$ contains precisely two $1$s. \end{lemma}

For the proof of Lemma \ref{lem 2 1s}, note first that $|\mx{i}| = \#1$s in $\mx{i} + 2\cdot\#2$s in $\mx{i}$, and so for there to be {\em any} pairings, the number of $1$s must be even.   Suppose, then, that $u,v,w\in\{1,\ldots,|\mx{i}|\}$ are distinct elements at positions corresponding to $1$s in $\mx{i}$.  The condition $\pi\vee\widehat{0_n} = 1_{|\mx{i}|}$ implies that there are paths
\[ \begin{aligned} v &= u_1 \sim_{\s_1} u_2 \sim_{\s_2} \cdots\sim_{\s_{r-2}} u_{r-1} \sim_{\s_{r-1}} u_r = u, \\
v &= w_1 \sim_{\t_1} w_2 \sim_{\t_2} \cdots \sim_{\t_{s-2}} w_{s-1} \sim_{\t_{s-1}} w_s = w, \end{aligned} \] 
where the sequences $\s_j$ and $\t_j$ alternate between $\pi$ and $\widehat{0_n}$.  By assumption, $v$ corresponds to a singleton in $\widehat{0_n}$, and so (assuming that the paths are ``minimal'' so that no number appears as two differently indexed $u_j$ or $v_j$) we must have $\s_1 = \t_1 = \pi$.  But $\pi$ is a pairing, so there is a unique point $v_2$ with $u_1 = w_1 \sim v_2$, and therefore $u_2 = w_2 = v_2$.  Then $v_2 \sim_{\s_2} u_3$ and $v_2 \sim_{\t_2} w_3$, where $\s_2 = \t_2 = \widehat{0_n}$.  If $v_2$ is a singleton in $\widehat{0_n}$, then this is the end of both paths, meaning $u = u_2 = w_2 = w$, contradicting our assumption.  Otherwise, the block of $\widehat{0_n}$ containing $v_2$ is a $2$-block, in which case there is a {\em unique} $v_3$ with $v_2\sim_{\widehat{0_n}} v_3$, and so $u_3 = w_3 = v_3$.  Continuing inductively, we reach a contradiction to the fact that $u\ne w$.  Hence, there must be precisely two $1$s in $\mx{i}$.

\begin{remark} \label{rk two ends} The above proof is really just the following trivial observation: the singletons in $\widehat{0_n}$ must be ends of a path joining blocks, and a path can have only $0$ or $2$ ends, thence $\widehat{0_n}$ can have only $0$ or $2$ $1$s if it is to have this path-connected property. \end{remark}

Lemma \ref{lem 2 1s} shows that the only contributing $\mx{i}$ to Equation \ref{eq expand k} are those of the form $2^{\e_1}\, 1\, 2^{\e_2} \, 1 \, 2^{\e_3}$ for some $\e_1,\e_2,\e_3\ge 0$.  The next Lemma shows that either $\e_2=0$ or $\e_1 = \e_3 = 0$ in order for the term to contribute.

\begin{lemma} \label{lem adjacent 1s} Suppose $\mx{i}$ is a length $n$ string containing both $1$s and $2$s, and let $\widehat{0_n}$ denote the corresponding interval partition.  If $\mx{i}$ contains the substring $(1,2,1,2)$ or $(2,1,2,1)$, then $\kk_n[\alpha_\mx{i}] = 0$.
\end{lemma}

For the proof of Lemma \ref{lem adjacent 1s}, note that the discussion following Equation \ref{eq sum 2} may be applied locally, and so in a string of the form $\mx{i}=2^{\e_1}\, 1\, 2^{\e_2} \, 1 \, 2^{\e_3}$ where $\e_1,\e_2,\e_3>0$, any pairing $\pi$ that connects the blocks of $\widehat{0_n}$ must pair according to the dark lines of Figure \ref{fig proof 4}.
\begin{figure}[htbp]
\begin{center}
\unitlength=4.5pt
    \begin{picture}(95, 23)(0,-12)
    \thinlines
    \gasset{Nw=1, Nh=1, Nmr=1,Nframe=n,Nfill=y, Nmarks=n, ExtNL=y, NLdist=2,AHnb=0}
    \node(s1)(0,0){}
    \node(s2)(5,0){}
    \drawrect[dash={0.5 1}0,Nfill=n,Nframe=y](-2,-2,7,2)
    \node(s3)(20,0){}
    \drawedge[dash={0.2 1.5}0, sxo=4,exo=-4](s2,s3){} 
    \node(s4)(25,0){$\hspace{4pt}c^\ast$}
    \drawrect[dash={0.5 1}0,Nfill=n,Nframe=y](18,-2,27,2)
    \node(s5)(30,0){}
    \drawrect[dash={0.5 1}0,Nfill=n,Nframe=y](28,-2,32,2)

    \drawedge[curvedepth=3,sxo=-1,syo=6,exo=1,eyo=6](s1,s4){$\e_1$}    
    
    \node(s6)(35,0){}
    \node(s7)(40,0){}
    \drawrect[dash={0.5 1}0,Nfill=n,Nframe=y](33,-2,42,2)
    \node(s8)(55,0){}
    \drawedge[dash={0.2 1.5}0, sxo=4,exo=-4](s7,s8){} 
    \node(s9)(60,0){$\hspace{4pt}c^\ast$}
    \drawrect[dash={0.5 1}0,Nfill=n,Nframe=y](53,-2,62,2)
    \node(s10)(65,0){}
    \drawrect[dash={0.5 1}0,Nfill=n,Nframe=y](63,-2,67,2)

    \drawedge[curvedepth=3,sxo=-1,syo=6,exo=1,eyo=6](s6,s9){$\e_2$}
    
\node(s11)(70,0){}
    \node(s12)(75,0){}
    \drawrect[dash={0.5 1}0,Nfill=n,Nframe=y](68,-2,77,2)
    \node(s13)(90,0){}
    \drawedge[dash={0.2 1.5}0, sxo=4,exo=-4](s12,s13){} 
    \node(s14)(95,0){}
    \drawrect[dash={0.5 1}0,Nfill=n,Nframe=y](88,-2,97,2)

    \drawedge[curvedepth=3,sxo=-1,syo=6,exo=1,eyo=6](s11,s14){$\e_3$}    

    \drawline[linewidth=0.25](0,-2.8)(0,-12)(95,-12)(95,-2.8)
    \drawline[linewidth=0.25](5,-2.8)(5,-6)(10,-6)(10,-2.8)
    \drawline[linewidth=0.25](15,-2.8)(15,-6)(20,-6)(20,-2.8)
    \drawline[linewidth=0.25](40,-2.8)(40,-6)(45,-6)(45,-2.8)
    \drawline[linewidth=0.25](50,-2.8)(50,-6)(55,-6)(55,-2.8)
    \drawline[linewidth=0.25](75,-2.8)(75,-6)(80,-6)(80,-2.8)
    \drawline[linewidth=0.25](85,-2.8)(85,-6)(90,-6)(90,-2.8)

    \drawline[linewidth=0.05](30,-2.8)(30,-6)(35,-6)(35,-2.8)
    \drawline[linewidth=0.05](65,-2.8)(65,-6)(70,-6)(70,-2.8)
    \drawline[linewidth=0.05](25,-2.8)(25,-9)(60,-9)(60,-2.8)

     \end{picture}
\caption{\label{fig proof 4}}
\end{center}
\end{figure}

\noindent The two singletons cannot pair together, since they would then be isolated by $\widehat{0_n}$. There are thence four positions where the two singletons may pair.  If the two singletons pair to the $\e_1$ and $\e_3$ blocks, then the $\e_2$ block is isolated, hence at least one singleton must pair to the $\e_2$ block, and then the other must pair outside the $\e_2$ block (or again that block would be isolated).  Without loss of generality, suppose that the first singleton pairs to the $\e_2$ block (otherwise we could simply reflect the figure).  It must therefore pair to the adjacent position (or else this position cannot pair anywhere without a crossing).   The remaining singleton must pair to its adjacent block, $\e_3$, for otherwise the right-most open slot in the $\e_2$-block could not be paired without crossings.  These pairings are represented in the light lines in Figure \ref{fig proof 4}.  This forces the remaining pairing in $\pi$  (between the right-most points in the $\e_1$ and $\e_2$ blocks) to match a $c^\ast$ with a $c^\ast$, resulting in a $0$ cumulant.  Therefore, although this pairing does satisfy the connectedness condition $\pi\vee\widehat{0_n} = 1_{|\mx{i}|}$, the cumulant $\kk_n[\alpha_{|\mx{i}|}]=0$.

\medskip

\noindent Hence, at least one of $\e_1,\e_2,\e_3$ must be $0$.  If either $\e_1$ or $\e_3$ is $0$ while the other two are $>0$, the above argument (unchanged) gives the same result.  So consider the case that $\e_2=0$ while $\e_1,\e_3>0$, represented in Figure \ref{fig proof 5}.  The local argument above Figure \ref{fig proof 4} yields the necessity of the dark lines here.
\begin{figure}[htbp]
\begin{center}
\unitlength=4.5pt
    \begin{picture}(65,10)(0,-10)
    \thinlines
    \gasset{Nw=1, Nh=1, Nmr=1,Nframe=n,Nfill=y, Nmarks=n, ExtNL=y, NLdist=2,AHnb=0}
    \node(s1)(0,0){}
    \node(s2)(5,0){}
    \drawrect[dash={0.5 1}0,Nfill=n,Nframe=y](-2,-2,7,2)
    \node(s3)(20,0){}
    \drawedge[dash={0.2 1.5}0, sxo=4,exo=-4](s2,s3){} 
    \node(s4)(25,0){}
    \drawrect[dash={0.5 1}0,Nfill=n,Nframe=y](18,-2,27,2)
    \node(s5)(30,0){}
    \drawrect[dash={0.5 1}0,Nfill=n,Nframe=y](28,-2,32,2)
    \node(s6)(35,0){}
    \drawrect[dash={0.5 1}0,Nfill=n,Nframe=y](33,-2,37,2) 
    \node(s7)(40,0){}
    \node(s8)(45,0){}
    \drawrect[dash={0.5 1}0,Nfill=n,Nframe=y](38,-2,47,2)
    \node(s9)(60,0){}
    \drawedge[dash={0.2 1.5}0, sxo=4,exo=-4](s8,s9){} 
    \node(s10)(65,0){}
    \drawrect[dash={0.5 1}0,Nfill=n,Nframe=y](58,-2,67,2)
    
    \drawline[linewidth=0.25](0,-2.8)(0,-9)(65,-9)(65,-2.8)
    \drawline[linewidth=0.25](5,-2.8)(5,-6)(10,-6)(10,-2.8)
    \drawline[linewidth=0.25](15,-2.8)(15,-6)(20,-6)(20,-2.8)
    \drawline[linewidth=0.25](45,-2.8)(45,-6)(50,-6)(50,-2.8)
    \drawline[linewidth=0.25](55,-2.8)(55,-6)(60,-6)(60,-2.8)

    \drawline[linewidth=0.05](25,-2.8)(25,-6)(30,-6)(30,-2.8)
    \drawline[linewidth=0.05](35,-2.8)(35,-6)(40,-6)(40,-2.8)
     
    \end{picture}
\caption{\label{fig proof 5}}
\end{center}
\end{figure}

\noindent Since the two singletons cannot pair together (as that block would be isolated), there is only one non-crossing pairing, given by the light lines in Figure \ref{fig proof 5}.  Each of the dark lines gives a contribution $\kk_2[c^\ast,c]$ (or $\kk_2[c,c^\ast]$ for the outside pairing), yielding $1$.  The remaining pairings are $\kk_2[c^\ast,\alpha_1]$ and $\kk_2[\alpha_1,c]$.  Hence, the index $\mx{i} = 2^{\e_1}\,1\,1\,2^{\e_3}$ with $\e_1,\e_3>0$ yields $\kk_n[\alpha_\mx{i}] = \kk_2[c^\ast,\alpha_1]\cdot\kk_2[\alpha_1,c]$, which is non-zero (we calculate it below).

\medskip

\noindent Finally, we consider the case that only one of $\e_1,\e_2,\e_3$ is non-zero.  (The case that all three vanish means that $n=2$, and we will consider that case separately at the end.)  Each of these three cases is really just a rotation of the one non-zero contributing case above: that is, a cyclic permutation of the string $\mx{i} = 2^{\e_1}\,1\,1\,2^{\e_3}$ already considered.  There are a total of $n$ such permutations, and each contributes the same cumulant (this follows from the fact that cyclic permutations induce lattice-isomorphisms of $NC(|\mx{i}|)$). Hence, each of these $n$ contributes the term $\kk_2[c^\ast,\alpha_1]\cdot\kk_2[\alpha_1,c]$.  This completes the proof of Lemma \ref{lem adjacent 1s}.

\medskip

Let us now collect all terms contributing to Equation \ref{eq expand k}.  For $n>2$, Equation \ref{eq endpoint 1}  yields that if $\mx{i}$ contains only $1$s then there is no contribution to the $n$th cumulant, and Equation \ref{eq endpoint 2} gives a contribution of $1$ in the case that $\mx{i}$ contains only $2$s.  Lemmas \ref{lem 2 1s} and \ref{lem adjacent 1s} then show that among all other $\mx{i}$, only $n$ contribute a non-zero cumulant, each equal to the product $\kk_2[c^\ast,\alpha_1]\cdot\kk_2[\alpha_1,c]$, which we now calculate:
\[ \begin{aligned} \kk_2[c^\ast,\alpha_1] &= \kk_2[c^\ast,-\ll(c+c^\ast)] = -\ll\left(\kk_2[c^\ast,c]+\kk_2[c^\ast,c^\ast]\right) = -\ll \\
\kk_2[\alpha_1,c] &= \kk_2[-\ll(c+c^\ast),c] = -\ll\left(\kk_2[c,c] + \kk_2[c^\ast,c]\right) = -\ll. \end{aligned} \]
Hence, the total contribution is
\begin{equation} \label{eq nth cumulant} \kk_n[\alpha_1+\alpha_2,\ldots,\alpha_1+\alpha_2] = 1+n\ll^2, \quad n>2. \end{equation}
For $n=1$, $\kk_1$ is the mean; $\kk_1[c] = \kk_1[c^\ast]  =0$, and so $\kk_1[\alpha_1+\alpha_2] = \kk_1[cc^\ast] =1$.  The second cumulant nearly fits into the above analysis, but can be handled separately more easly:
\[ \kk_2[\alpha_1+\alpha_2,\alpha_1+\alpha_2] = \kk_2[\alpha_1,\alpha_1] + \kk_2[\alpha_1,\alpha_2] + \kk_2[\alpha_2,\alpha_1] + \kk_2[\alpha_2,\alpha_2]. \]
The two middle terms are odd, and since odd cumulants of $c,c^\ast$ are $0$, these cumulants are $0$.  The first and last terms are included in Equations \ref{eq endpoint 1} and \ref{eq endpoint 2}, and yield $2\ll^2$ and $1$, respectively.
Thus, from Equation \ref{eq c R-transform 2} we have
\[ \mathscr{R}_{\alpha_1+\alpha_2}(z) = 1 + (1+2\ll^2)z + \sum_{n\ge 3} (1+n\ll^2)z^{n-1}, \]
and so
\[ \mathscr{R}_{|\ll-c|^2}(z) = \ll^2 + 1 + (1+2\ll^2)z + \sum_{n\ge 3} (1+n\ll^2)z^{n-1} = \sum_{n\ge 1}(1+n\ll^2)z^{n-1}, \]
yielding the power-series expansion of Equation \ref{eq circular R} as required.
\end{proof}

\subsection{The support of the spectral measure of $|\ll-c|^2$} \label{sect circular support}

Denote by $K_{|\ll-c|^2}$ the function
\[ K_{|\ll-c|^2}(z) = \mathscr{R}_{|\ll-c|^2}(z)+1/z.\]
From Equation \ref{eq G-R functional eqn}, $G_{\mu_\ll}(K_{|\ll-c|^2}(z)) = z$ for small $z\in\C_+$, where $G_{\mu_\ll}$ is the Cauchy transform of the spectral measure $\mu_\ll$ of $|\ll-c|^2$.  The result of Theorem \ref{thm circular R} yields
\begin{equation} \label{eq K} K_{|\ll-c|^2}(z) = \frac{1}{z} + \frac{1}{1-z} + \frac{\ll^2}{(1-z)^2} = \frac{1+(\ll^2-1)z}{z(1-z)^2}. \end{equation}
For notational convenience, let
\begin{equation} \label{eq m} m = \ll^2-1 \qquad K_m \equiv K_{|\ll-c|^2} \qquad G_m = G_{\mu_\ll}. \end{equation}
So $G_m\circ K_m(z) = z$ for small $z\in\C_+$.  Our goal is to determine the support of the measure $\mu_\ll$.  Note that this support set is precisely the set of singular points for the Cauchy transform $G_m$, which we now set out to determine.  The first derivative of $K_m(z)$ is given by
\begin{equation} \label{eq Km'} K_m'(z) = \frac{1-3z-2mz^2}{z^2(z-1)^3}. \end{equation}
The quadratic polynomial in the numerator has two zeroes,
\[ z^{\pm} = \frac{-3\pm \sqrt{9+8m}}{4m}. \]
Since $m=\ll^2-1>0$ it is easy to check that
\[ z^-\in(-\infty,0) \quad \text{and} \quad z^+\in(0,1). \]
Moreover, by factoring the polynomial in the numerator of Equation \ref{eq Km'}, one has
\[ K_m'(z) = -\frac{2m(z-z^-)(z-z^+)}{z^2(z-1)^3} \]
which shows that
\[ K_m'(z)<0 \quad \text{for} \quad z\in(z^-,0)\cup(0,z^+). \]
Thus $K_m'$ is a strictly decreasing function on each of the intervals $(z^-,0)$ and $(0,z^+)$.  Now, set $s^\pm \equiv K_m(z^\pm)$; then simple (though tedious) calculation yields
\begin{equation} \label{eq s} %\begin{aligned}
s^\pm = %\frac{(\sqrt{9+8m}\pm 3)^3}{8(\sqrt{9+8m}\pm 1)} &= \frac{(\sqrt{9+8m}\pm 3)^3(\sqrt{9+8m}\mp 1)}{64(m+1)}  \\
%&=
\frac{27+36m+8m^2\pm(9+8m)^{3/2}}{8(m+1)}.
%\end{aligned}
\end{equation}
Since $\lim_{z\to 0\pm} K_m(z) = \pm\infty$, $K_m$ is a decreasing bijection of $(z^-,0)$ onto $(-\infty,s^-)$ and $K_m$ is also a decreasing bijection of $(0,z^+)$ onto $(s^+,\infty)$.  Moreover since $s^-<s^+$ by Equation \ref{eq s}, $K_m$ is a bijection of $(z^-,0)\cup(0,z^+)$ onto $(-\infty,s^-)\cup(s^+,\infty)$.  Let
\[ L_m\colon (-\infty,s^-)\cup(s^+,\infty)\to(z^-,0)\cup(0,z^+) \]
denote the (function) inverse of the above restriction of $K_m$.  Then $L_m(w) = G_m(w)$ for large values of $|w|$.  Moreover, $L_m$ is real analytic, and can therefore be extended to a complex analytic function $\tilde{L}_m$ in a complex neighbourhood  $U$ of $(-\infty,s^-)\cup(s^+,\infty)$, which we can assume has only two connected components $U^-\supset (-\infty,s^-)$ and $U^+\supset (s^+\infty)$.  By uniqueness of analytic continuation to open connected sets, it follows that $\tilde{L}_m(w) = G_m(w)$ for all $w\in U$.  In particular, $G_m$ has no singular points in $(-\infty,s^-)\cup(s^+,\infty)$.  Ergo, it follows that
\begin{equation} \label{eq [s-,s+]} \supp \mu_\ll \subseteq [s^-,s^+]. \end{equation}
Since $K_m'(z^\pm)=0$, the graph of $L_m$ has vertical tangents at the endpoints $(s^-,z^-)$ and $(s^+,z^+)$.  Therefore $s^\pm$ are both singular points for $G_m$, and we conclude that
\begin{equation} \label{eq s+-} s^\pm\in \supp\mu_\ll. \end{equation}
To prove that $\supp\mu_\ll = [s^-,s^+]$, we apply the result of Voiculescu \cite{Voiculescu97} which implies that the unital $C^\ast$-algebra generated by a semicircular family $(s_j)_{j\in J}$ has no non-trivial projections.  Therefore the spectrum $\mathrm{spec}(x)$ of any selfadjoint element $x\in C^\ast\left(\{s_j\,;\,j\in J\}\cup\{1\}\right)$ is connected, and therefore is either an interval or  single point. The standard circular operator $c$ is equal to
\[ c = \frac{1}{\sqrt{2}}(s_1+is_2) \]
where $s_1 = \frac{1}{\sqrt{2}}(c+c^*),s_2 = \frac{1}{\sqrt{2}i}(c-c^*)$ is a semicircular family with two elements.  Hence,
\[ \supp\mu_\ll = \mathrm{spec}\left((\ll-c)^*(\ll-c)\right) \]
is either an interval or a single point; it now follows from Equations \ref{eq [s-,s+]} and \ref{eq s+-} that $\supp \mu_\ll = [s^-,s^+]$.  In particular, $s^-$ is the infimum of support of $\mu_\ll$.  Substituting $\ll^2-1$ for $m$ in Equation \ref{eq s}, we have the following.

\begin{proposition} \label{prop inf spec} Let $c$ be a standard circular operator, and let $\ll>1$.  Then\begin{equation} \label{eq inf spec} \inf\mathrm{spec}\,|\ll-c|^2 = s^- = \frac{8\ll^4+20\ll^2-1-(8\ll^2+1)^{3/2}}{8\ll^2}. \end{equation}
%$s^-(\ll^2-1) = \frac{1}{8\ll^2}\left[-1+20\ll^2+8\ll^4-\sqrt{1+24\ll^2+192\ll^4+512\ll^6}\right].$ 
\end{proposition}

%\begin{remark} A similar analysis to the one above shows that $\|\ll-c\|^2 = \| |\ll-c|^2 \| = \sup\mathrm{spec}\,|\ll-c|^2  = \sup\supp\mu_\ll = s^+(\ll^2-1)$.  However, we are only concerned with the infimum here. \end{remark}

%\begin{remark} \label{rk critical values} Since the measure $\mu_\ll$ is $0$ outside the interval $[s^-,s^+]$ but is non-zero at those boundary points, the Cauchy transform $G_{\mu_\ll}(z)$ is singular as $z$ approaches $s^{\pm}$.  This means that $s^{\pm}$ are critical values for the inverse $K_{|\ll-c|^2}(z) = (1+mz)/z(1-z)^2$, and it is indeed easy to check that these are its {\em only} critical values.  However, a more complicated analytic continuation argument is required to show from this fact alone that $\supp\mu_\ll = [s^-,s^+]$, and the direct argument above is simpler in this case.  In Section \ref{sect general case}, where direct calculation is not possible, we will develop these other means.
%\end{remark}

\noindent Proposition \ref{prop inf spec} yields an exact formula for the norm $\|R_c(\ll)\| = \|(\ll-c)^{-1}\|$: it is the reciprocal of the square root of the expression in Equation \ref{eq inf spec}, as discussed following Equation \ref{eq T}.  We are primarily concerned with the leading order terms in this expression.  It is easy to calculate the Taylor expansion of the function in Equation \ref{eq inf spec}.  The result is
\begin{equation} \label{eq taylor} \inf\mathrm{spec}|\ll-c|^2 = \frac{32}{27}(\ll-1)^3 + O((\ll-1)^4). \end{equation}
Taking the reciprocal square root, and noting that $\|c\|_2=1$ and $\|c\|_4^4 = 2$ so that $\upsilon(c) = 1$, Equation \ref{eq taylor} proves Theorem \ref{main theorem} in the special case that the $\mathscr{R}$-diagonal operator $a$ is a circular $c$.   

\medskip

\noindent It is possible to compute the Cauchy transform $G_m$ of $\mu_\ll$ explicitly using Cardano's formula for solving cubic equations: for $w\in\C-\R$, the number $z=G_m(w)$ is a solution to the equation $K_m(z)=w$, which can be reduced to the following cubic equation in $z$:
\[ z^3 - 2z^2 + \left(1-\frac{m}{w}\right)z - \frac{1}{w} = 0. \]
After determining the correct branch among the three solutions, one can then use the Stieltjes inversion formula of Equation \ref{eq Stieltjes inversion} to show that $\mu_\ll$ has a density with respect to Lebesgue measure, and compute this density explicitly.  Figures \ref{fig densities} and \ref{fig inverse densities} below are based on such computations.
\begin{figure}[htbp]
\begin{center}
\includegraphics[scale=0.27]{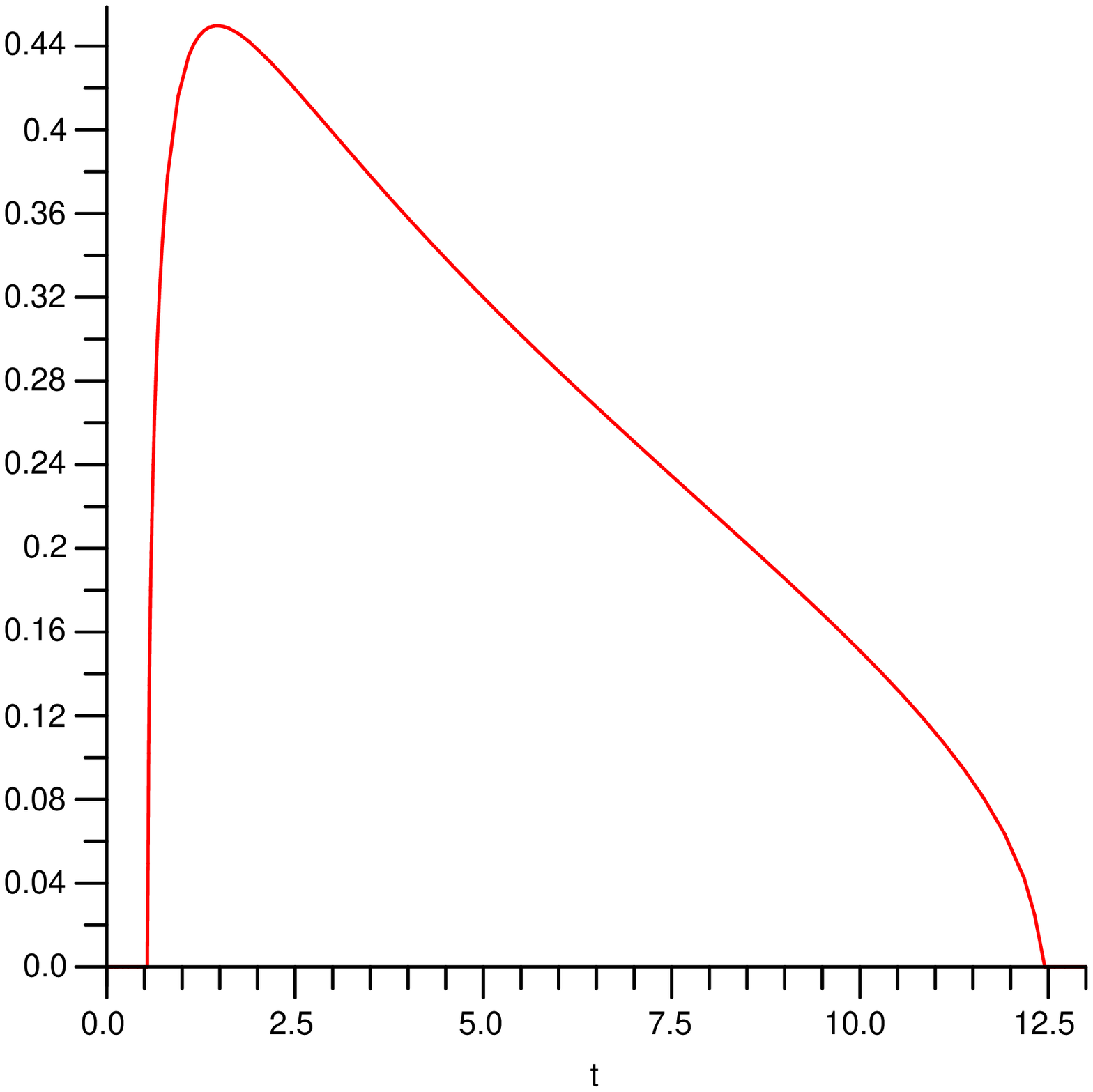}
\includegraphics[scale=0.27]{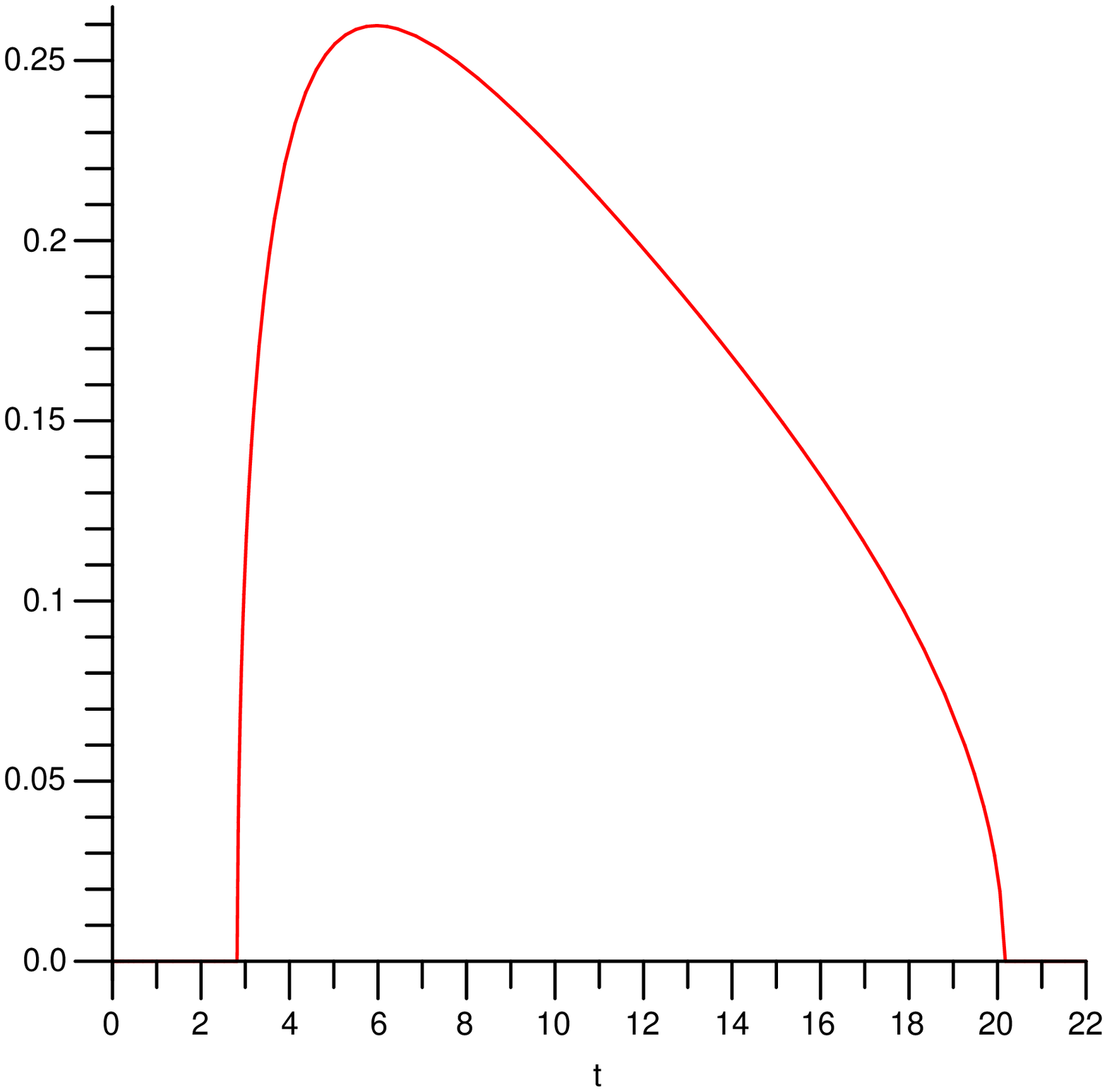}
\includegraphics[scale=0.27]{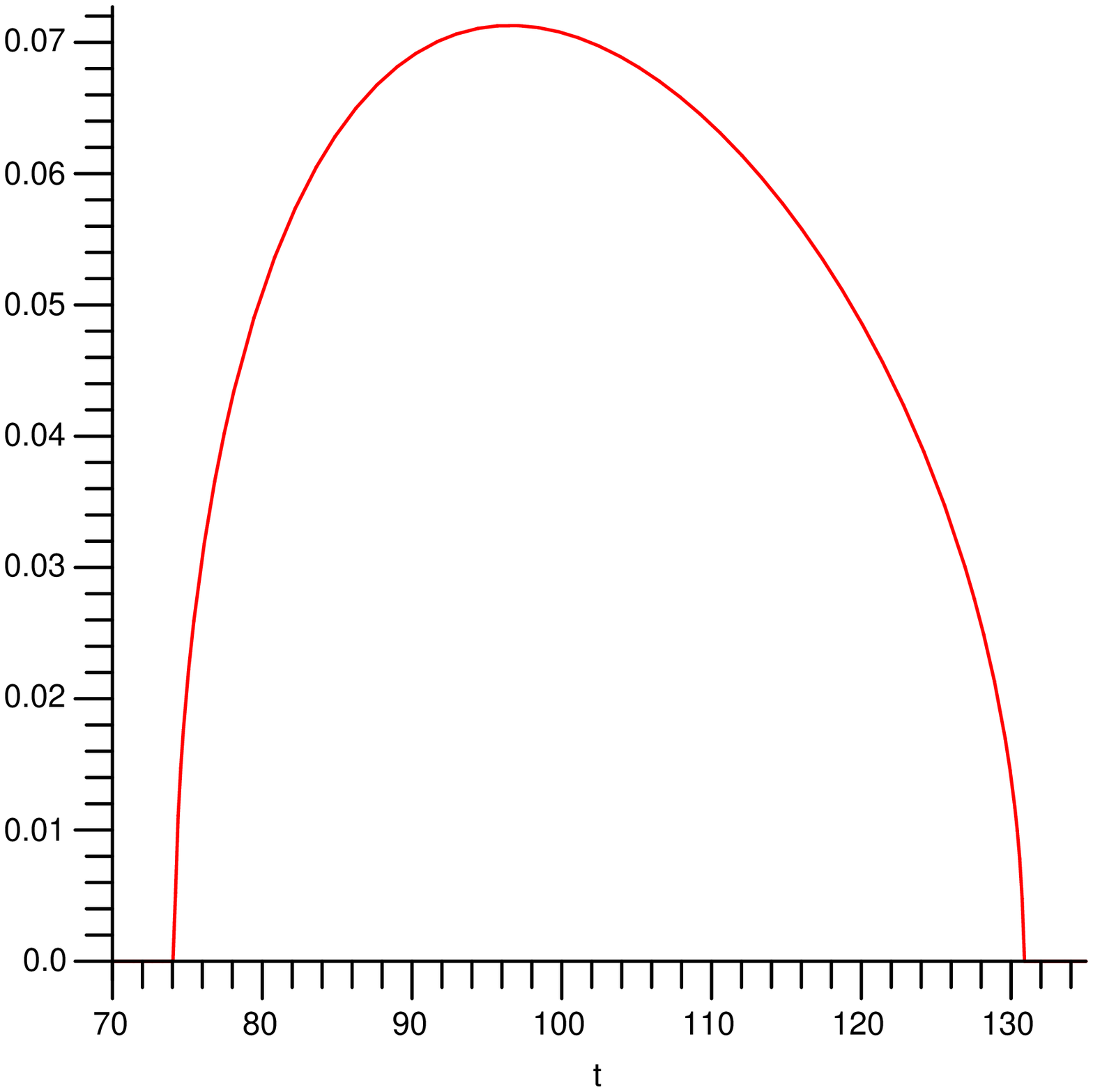}
\caption{\label{fig densities} The densities of the measures $\mu_\ll(dt)$, for $\ll=2$, $3$, and $10$, from left to right.}
\end{center}
\end{figure}

\begin{figure}[htbp]
\begin{center}
\includegraphics[scale=0.27]{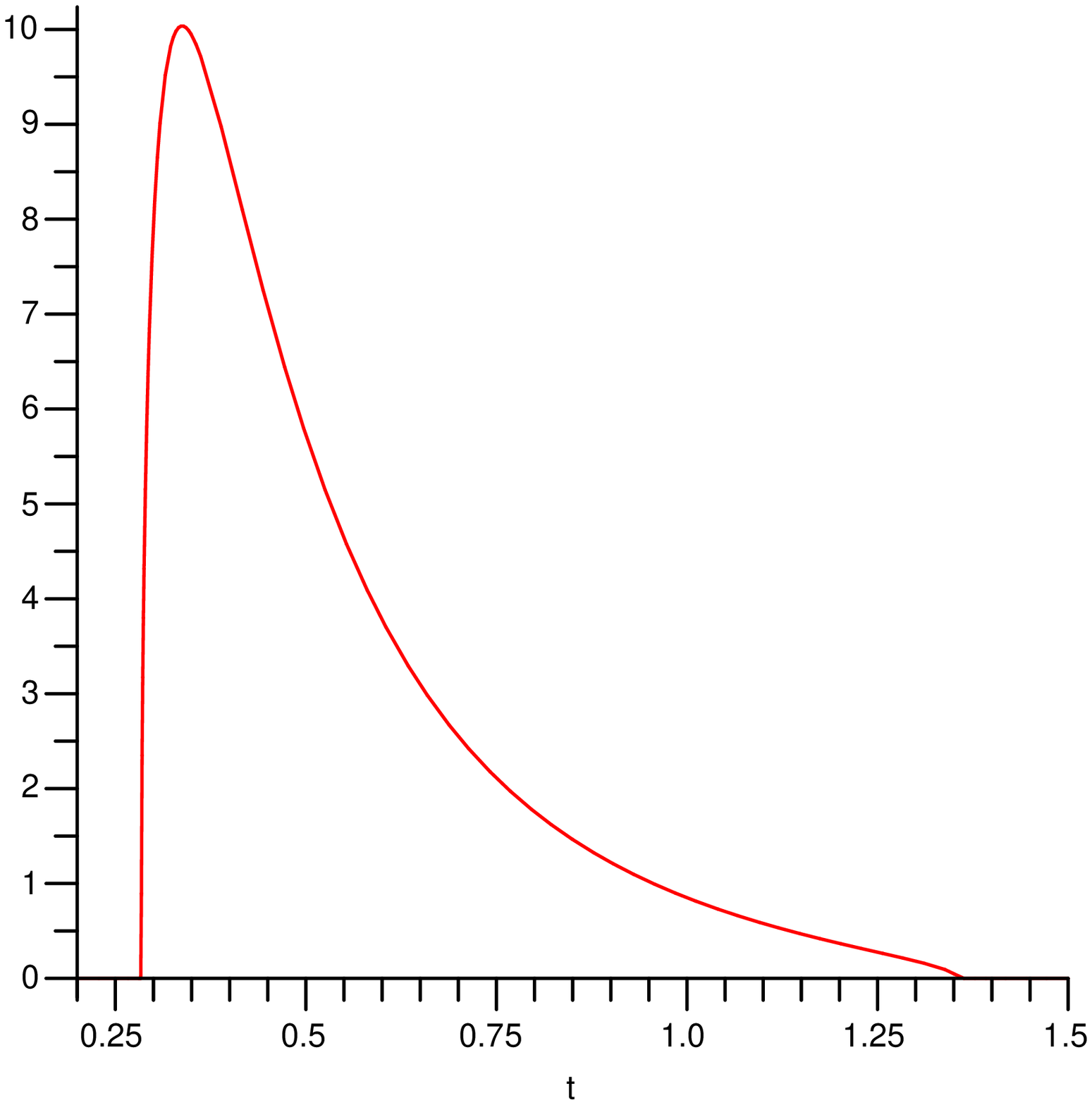}
\includegraphics[scale=0.27]{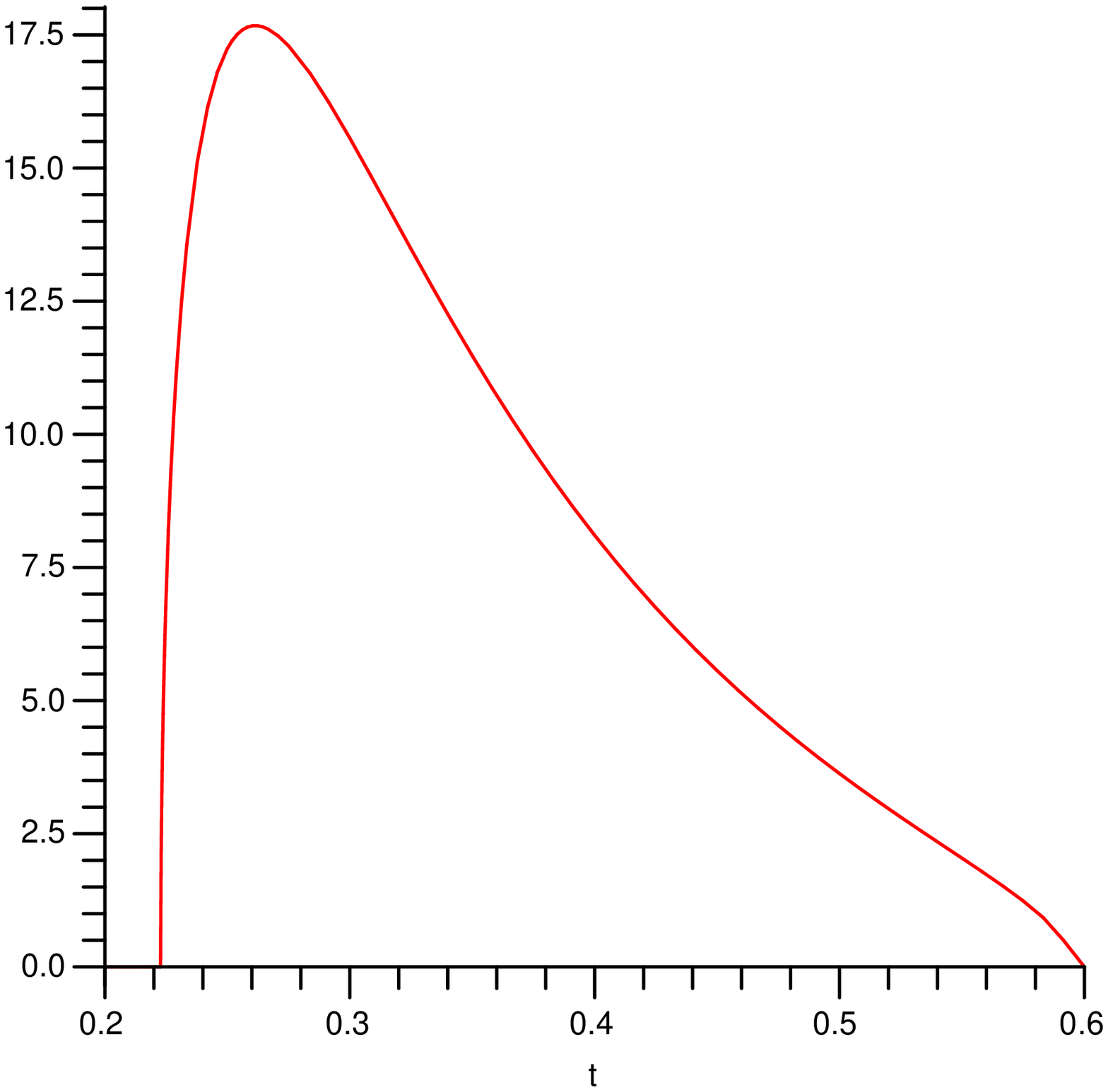}
\includegraphics[scale=0.27]{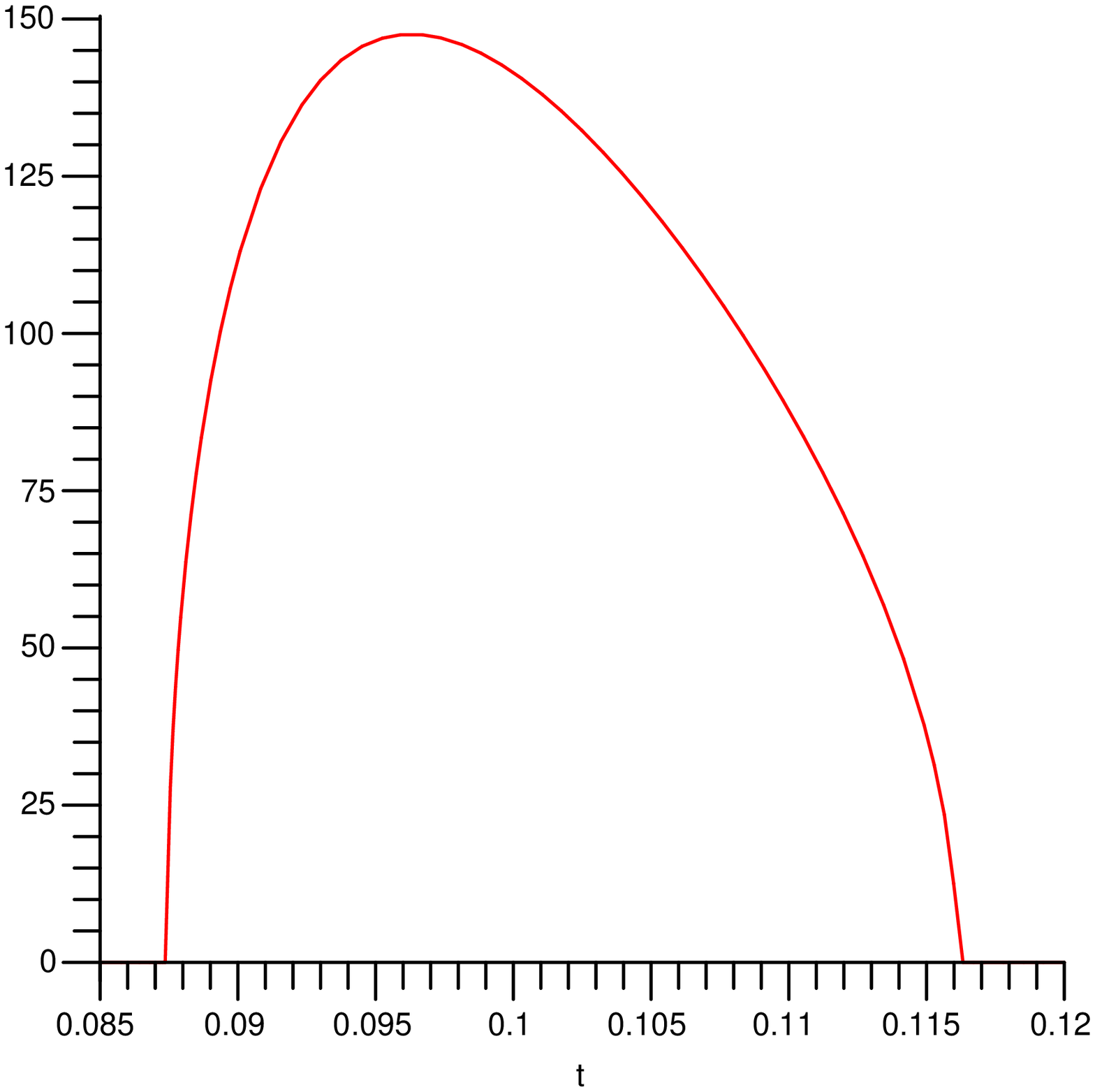}
\caption{\label{fig inverse densities} The densities of the spectral measures $(1/\sqrt{\cdot}\,)_\ast\mu_\ll(dt)$ of the operators $|\ll-c|^{-1}$ for $\ll=2$, $3$, and $10$, from left to right.  Note that $|R_c(\ll)|^2 = [(\ll-c^\ast)(\ll-c)]^{-1}$ has the same distribution as $[(\ll-c)(\ll-c^\ast)]^{-1} = |\ll-c|^{-2}$ since the von Neumann algebra generated by $c$ is tracial; hence the density of $|\ll-c|^{-1}$ is the same as the density of $|R_c(\ll)|$.}
\end{center}
\end{figure}

\section{Resolvents in the General $\mathscr{R}$-Diagonal Case} \label{sect general case}

This section is largely devoted to the proof of Theorem \ref{main theorem}.  The direct calculations used in Section \ref{sect circular} are not available in this case: for a general $\mathscr{R}$-diagonal operator $a$, it is far more difficult to find a closed-formula for the $\mathscr{R}$-transform of $|\ll-a|^2$.
% and even if one existed a near-direct calculation of the Cauchy transform as we did in Section \ref{sect circular support} is in general impossible.
Nevertheless, in Section \ref{sect circular support}, we determined the support of the measure through critical points, following \cite{Haagerup Schultz}; through a similar approach, we will be able to calculate the norm of $R_a(\ll)$ to leading order as it tends to $\infty$, for any $\mathscr{R}$-diagonal operator $a$.

\medskip

Let $a$ be $\mathscr{R}$-diagonal in a $\mathrm{II}_1$-factor $\mathscr{A}$.  Expanding $\mathscr{A}$ if necessary, we may choose a Haar unitary $u$ $\ast$-free from $a$.  It is easy to check that $au$ has the same $\ast$-distribution as $a$; indeed, this can be used as a definition for $\mathscr{R}$-diagonality (cf. \cite{Nica Speicher Book}).  As such, for $\ll>0$,
\begin{equation} \label{eq same distribution} |a-\ll|^2 \mathop{\sim}_{\ast\mathcal{D}} |au-\ll|^2 = (au-\ll)(u^\ast a^\ast - \ll) = (a-\ll u^\ast) uu^\ast(a^\ast-\ll u) = |a-\ll u^\ast|^2. \end{equation}
Hence, the spectral measure of $|a-\ll|$ is the same as that of $|a-\ll u^\ast|$, and $(a,\ll u^\ast)$ are $\ast$-free $\mathscr{R}$-diagonal elements.  We may now employ the following tool for calculating the $\mathscr{R}$-transform of a sum.

\begin{proposition} Let $a,b$ be $\ast$-free $\mathscr{R}$-diagonal elements, and for any self-adjoint element $x$ let $\mu_x$ denote its spectral measure.  For any Borel probability measure $\mu$ on $\R$, denote by $\tilde{\mu}$ the symmetrization of $\mu$: for any Borel set $B\subseteq\R$, $\tilde{\mu}(B) = \frac{1}{2}(\mu(B) + \mu(-B))$.  Then
\begin{equation} \label{eq symmetric sum} \tilde{\mu}_{|a+b|} = \tilde{\mu}_{|a|} \boxplus \tilde{\mu}_{|b|}. \end{equation}
\end{proposition}

\begin{proof} This is proved in Proposition 5.2 in \cite{Nica Speicher Duke Paper}.  A different proof is given in Proposition 3.5 in \cite{Haagerup Larsen}.
\end{proof}

Applying Equation \ref{eq symmetric sum} to the preceding discussion, we have
\[ \tilde{\mu}_{|a-\ll|} = \tilde{\mu}_{|a - \ll u^\ast|} = \tilde{\mu}_{|a|}\boxplus\tilde{\mu}_{|\ll u^\ast|}. \]
Of course, $|\ll u^\ast| = \ll$ for $\ll>0$, and so $\tilde{\mu}_{|\ll u^\ast|} = \frac{1}{2}(\delta_{\ll} + \delta_{-\ll})$.  Therefore, the associated Cauchy transform is $\frac{1}{2}(\frac{1}{z-\ll} + \frac{1}{z+\ll}) = \frac{z}{z^2-\ll^2}$.  Solving the equation $w=\frac{z}{z^2-\ll^2}$ for $z$ yields $z = \frac{1\pm\sqrt{1+4\ll^2 w^2}}{2w}$, and so Equation \ref{eq G-R functional eqn} yields that
\[ \mathscr{R}_{\tilde{\mu}_{|\ll u^\ast|}}(w) = \frac{1+\sqrt{1+4\ll^2 w^2}}{2w} - \frac{1}{w} = \frac{\sqrt{1+4\ll^2 w^2}-1}{2w} \]
(the sign of the square root is chosen so that $\mathscr{R}$ maps $\C_+$ into itself).  Employing, finally, the additivity of the $\mathscr{R}$-transform over free convolution, we have the following.

\begin{proposition} \label{prop R transform sum} Let $a$ be $\mathscr{R}$-diagonal.  Denote by $\mu$ the symmetrization $\tilde{\mu}_{|a|}$, and let $\mu_\ll$ denote the symmetrization $\tilde{\mu}_{|a-\ll|}$.  Then
\begin{equation} \label{eq R transform sum} \mathscr{R}_{\mu_\ll}(z) = \mathscr{R}_{\mu}(z) + \frac{\sqrt{1+4\ll^2 z^2}-1}{2z}, \quad \text{for small }z\in\C_+.\end{equation}
\end{proposition}

\begin{remark} \label{rk square} If a measure $\mu$ is supported in $[0,\infty)$, then $\mu$ and $\tilde{\mu}$ contain the same information.  Hence, Proposition \ref{prop R transform sum} actually allows the determination of the spectral measure of $|\ll-a|$. Despite this, a direct derivation of $G_\mu$ from $G_{\tilde{\mu}}$ is not obvious in general (the latter is the odd part of the former).  However, both have the same {\em square}.  That is, if $\square$ is the map $\square(x) = x^2$ for $x\in\R$, then the push-forwards $\square_\ast\mu = \square_\ast\tilde{\mu}$ are equal and so too are the Cauchy transforms in the case that $\mu$ is supported in $[0,\infty)$.
\end{remark}

\subsection{The $\mathscr{R}$-transform of $|\ll-c|^2$ -- analytic approach} \label{sect analytic approach to R}  To demonstrate the power of Equation \ref{eq R transform sum}, we will now use it to give an alternate, entirely analytic proof of Theorem \ref{thm circular R}.

\begin{proof}[Analytic proof of Theorem \ref{thm circular R}]
Let $\nu_\ll$ denote the spectral measure of $|\ll-c|^2 = (\ll-c)(\ll-c^\ast)$, and $\nu = \nu_0$ the spectral measure of $|c|^2 = cc^\ast$.  Following the notation in Proposition \ref{prop R transform sum} and Remark \ref{rk square}, with $\mu_\ll = \tilde{\mu}_{|\ll-c|}$ and $\mu = \tilde{\mu}_{|c|}$, we have $\nu_\ll = \square_\ast\mu_\ll$ and $\nu = \square_\ast\mu$ where $\square(x) = x^2$ for $x\in\R$.  Note that $\mu_{|c|}$ is the quarter-circular law $\mu_{|c|}(dt) = \frac{2}{\pi}\sqrt{4-t^2}\1_{[0,2]}(t)\,dt$, and so $\mu$ is the standard semicircle law $\mu(dt) = \frac{1}{\pi}\sqrt{4-t^2}\1_{[-2,2]}(t)\,dt$.  Thence $\mathscr{R}_\mu(z) = z$, and Equation \ref{eq R transform sum} reads
\begin{equation} \label{eq analytic proof 1} \mathscr{R}_{\mu_\ll}(z) = z + \frac{\sqrt{1+4\ll^2 z^2}-1}{2z}, \end{equation}
for small $z\in\C_+$.  In this case, the precise domain is easy to determine since $\mathscr{R}_\mu$ is analytic everywhere; the domain of analyticity above is $|z|<\frac{1}{2\ll}$.  Adding $1/z$, Equation \ref{eq G-R functional eqn} shows that the functional inverse $G_{\mu_\ll}^{\langle -1 \rangle}$ is given by
\[ G_{\mu_\ll}^{\langle -1 \rangle}(z) = z + \frac{\sqrt{1+4\ll^2 z^2}+1}{2z}, \quad 0<|z|<\frac{1}{2\ll}. \]
Put $s = \frac{\sqrt{1+4\ll^2 z^2}+1}{2z}$; then small $|z|$ corresponds to large $|s|$.  The quantity $s$ is best characterized as a solution to the quadratic equation $zs^2 - s - z\ll^2 = 0$, and so in terms of $s$ we have
\[ z = \frac{s}{s^2-\ll^2}, \]
which evidently makes sense for large $|s|$.  Thus, Equation \ref{eq analytic proof 1} may be written in the form
\begin{equation} \label{eq analytic proof 2} \frac{s}{s^2-\ll^2} = G_{\mu_\ll}\left(\frac{s}{s^2-\ll^2} + s\right), \quad \text{for }|s|\text{ large}.
\end{equation}
Now, let us consider $G_{\nu_\ll} = G_{\square_\ast \mu_\ll}$.  We have
\begin{equation} \begin{aligned} \label{eq analytic proof 3} G_{\nu_\ll}(w^2) = \int_{\R} \frac{1}{w^2-t}\, \square_\ast \mu_\ll(dt) &= \int_{\R} \frac{1}{w^2-t^2}\,\mu_\ll(dt) \\
&= \frac{1}{2w} \int_{\R} \left(\frac{1}{w-t}+\frac{1}{w+t}\right)\,\mu_\ll(dt). \end{aligned} \end{equation}
A change of variables and the fact that $\mu_\ll$ is symmetric shows that both integrands above have the same integral, and so Equation \ref{eq analytic proof 3} becomes
\begin{equation} \label{eq analytic proof 4}
G_{\nu_\ll}(w^2) = \frac{1}{w}\int_{\R} \frac{1}{w-t}\,\mu_\ll(dt) = \frac{1}{w}G_{\mu_\ll}(w). \end{equation}
Substitute $w=\frac{s}{s^2-\ll^2}+s$ into Equation \ref{eq analytic proof 2}, and Equation \ref{eq analytic proof 4} yields that for large $|s|$,
\begin{equation} \label{eq analytic proof 5} G_{\nu_\ll}\left(\left(\frac{s}{s^2-\ll^2}+s\right)^2\right) = \left(\frac{s}{s^2-\ll^2}+s\right)^{-1}\cdot\frac{s}{s^2-\ll^2}
= \frac{1}{1+s^2-\ll^2}. \end{equation}
Inverting $G_{\nu_\ll}$ in Equation \ref{eq analytic proof 5} and again using Equation \ref{eq G-R functional eqn}, we get for large $|s|$,
\begin{equation} \label{eq analytic proof 6} \begin{aligned} \mathscr{R}_{\nu_\ll}\left(\frac{1}{1+s^2-\ll^2}\right) &= G_{\nu_\ll}^{\langle -1 \rangle}\left(\frac{1}{1+s^2-\ll^2}\right) - (1+s^2-\ll^2) \\
&= \left(\frac{s}{s^2-a^2}+s\right)^2-(1+s^2-\ll^2) \\
&= \frac{s^2(1+s^2-\ll^2)^2}{(s^2-\ll^2)^2} - (1+s^2-\ll^2).
\end{aligned} \end{equation}
Finally, set $z = (1+s^2-\ll^2)^{-1}$, so that large $|s|$ corresponds to small $|z|$.  Then
\[ 1-z = \frac{s^2-\ll^2}{s^2-\ll^2+1}, \quad s^2 = \frac{1}{w}+\ll^2-1. \]
Substituting into Equation \ref{eq analytic proof 6}, we find that for small $z\ne 0$,
\begin{equation} \label{eq analytic proof 7} \begin{aligned}
\mathscr{R}_{\nu_\ll}(z) &= (z^{-1}+\ll^2-1)\cdot\frac{1}{(1-z)^2} - \frac{1}{z} \\
&= \frac{z^{-1}-1}{(1-z)^2} + \frac{\ll^2}{(1-z)^2} - \frac{1}{z} \\
&= \frac{1}{z(1-z)} - \frac{1}{z} + \frac{\ll^2}{(1-z)^2} = \frac{1}{1-z} + \frac{\ll^2}{(1-z)^2},
\end{aligned}\end{equation}
which is the desired result.
\end{proof}

\begin{remark} The above proof is rather shorter than the one in Section \ref{sect combinatorial R}.  It relies on the somewhat sophisticated analytic result of Proposition \ref{prop R transform sum}; on the other hand, the proof in Section \ref{sect combinatorial R} relies on the sophisticated combinatorial result of Theorem \ref{thm prod cumulant}.  The benefit of the combinatorial proof is that it provides a direct explanation for all of the terms in the the $\mathscr{R}$-transform of $\nu_\ll$, which is the reason we've included it. \end{remark}

\subsection{Analytic Continuation and Roots of $G_{\mu_\ll}$} \label{sect analytic continuation}

Our goal is to use Equation \ref{eq symmetric sum} to determine (to leading order) the smallest positive singular value of $G_{\mu_\ll}$, which is the reciprocal of the spectral radius of the resolvent $R_a(\ll) = (a-\ll)^{-1}$ of our $\mathscr{R}$-diagonal operator $a$ in the $\mathrm{II}_1$--factor $\mathscr{A}$.  Let $\phi$ denote the trace on $\mathscr{A}$.  Adding $1/z$ to both sides, rewrite Equation \ref{eq symmetric sum} in the form
\begin{equation} \label{eq Gmu} G_{\mu_\ll}\left(\mathscr{R}_\mu(z) + \frac{1+\sqrt{1+4\ll^2 z^2}}{2z}\right) = z, \quad \text{for small }z\ne 0. \end{equation}
Following Section 4 of \cite{Haagerup Schultz}, we introduce the auxiliary functions
\begin{equation}\begin{aligned}\label{eq h}
h(s) &= s\,\phi\left((aa^\ast + s^2)^{-1}\right), \qquad\qquad\qquad  s>0 \\
h_\ll(s) &= s\,\phi\left(((a-\ll)(a-\ll)^\ast + s^2)^{-1}\right), \;\; \ll,s>0.
\end{aligned}\end{equation}
Then, as proved in Lemma 4.2 in \cite{Haagerup Schultz},
\begin{align}
\label{eq G h} G_\mu(is) &= -ih(s), \quad \;\;\; s>0 \\
\label{eq Gl hl} G_{\mu_\ll}(is) &= -i h_\ll(s), \quad \ll,s>0.
\end{align}
Using Equation \ref{eq G h} together with the definition of $\mathscr{R}_\mu$ in terms of $G_\mu$, it follows that there is some large $s_0>0$ so that
\begin{equation} \label{eq R h}
\mathscr{R}_\mu(-ih(s)) - \frac{1}{i h(s)} = is, \quad s>s_0.
\end{equation}
Combining Equations \ref{eq Gmu} (with $z=-ih(s)$), \ref{eq Gl hl}, and \ref{eq R h}, we have
\begin{equation} \label{eq hl 1}
h_\ll\left(s-\frac{1}{2h(s)} + \frac{\sqrt{1-4\ll^2h(s)^2}}{2h(s)}\right) = h(s),
\end{equation}
for large $s$ --- say $s>s_\ll$ for some $s_\ll>s_0$.

\medskip

Based on Equation \ref{eq hl 1}, it was proved in Proposition 4.13 of \cite{Haagerup Schultz} (see also Lemma 4.8 and Definition 4.9) that $h_\ll(t)$ can be obtained from $h$ for all $t>0$ in the following way.

\begin{proposition} \label{HS prop 4.13}
For every $t>0$, the equation
\begin{equation} \label{eq HS prop 4.13 1} (s-t)\left(\frac{1}{h(s) - s + t}\right) = \ll^2 \end{equation}
has a unique solution $s = s(\ll,t)$ in the interval $(t,\infty)$, and 
\begin{equation} \label{eq HS prop 4.13 2} h_\ll(t) = h(s(\ll,t)), \quad t>0. \end{equation}
\end{proposition}

\begin{corollary} \label{cor h -} Let $s>0$, and assume that $1-4\ll^2 h(s)^2 \ge 0$ and
\[ s-\frac{1}{2h(s)} - \frac{\sqrt{1-4\ll^2h(s)^2}}{2h(s)} > 0. \]
Then
\begin{equation} \label{eq Cor h}
h_\ll\left(s-\frac{1}{2h(s)}-\frac{\sqrt{1-4\ll^2h(s)^2}}{2h(s)}\right) = h(s).
\end{equation}
\end{corollary}

\begin{proof} Put $t = s-\frac{1}{2s(h)} -\frac{\sqrt{1-4\ll^2 h(s)^2}}{2h(s)}$.  Then $s>t$ (since $h(s)>0$ from Equation \ref{eq h}).  It is a simple matter to check that $s$ is a solution to Equation \ref{eq HS prop 4.13 1}, and therefore $s=s(\ll,t)$.  Hence, Equation \ref{eq Cor h} follows from Equation \ref{eq HS prop 4.13 2}.
\end{proof}

\begin{remark} Comparing Equations \ref{eq hl 1} and \ref{eq Cor h}, we see that that point of Corollary \ref{cor h -} is that the {\em negative} root may also be chosen in the determining equation for $h_\ll$.  This results in the similar alternate version of Equation \ref{eq Gmu} in Proposition \ref{prop neg root} below, identifying what turns out to be the correct singular value of $G_{\mu_\ll}$.
\end{remark}

In the following, we assume that $\ll>\|a\|_2$.  Since $\mathrm{spec}(a) \subseteq B(0,\|a\|_2)$ (cf. \cite{Haagerup Larsen}), $a-\ll$ is invertible and therefore $\supp(\mu_\ll) \subseteq \R\setminus(-\d_\ll,\d_\ll)$ for some $\d_\ll>0$.
Therefore $G_{\mu_\ll}$ is defined and analytic in a complex neighbourhood of $0$.

\begin{proposition} \label{prop neg root} For all $z$ in a small complex neighbourhood of $0$,
\begin{equation} \label{eq neg root}
G_{\mu_\ll}\left(\mathscr{R}_\mu(z) + \frac{1-\sqrt{1+4\ll^2 z^2}}{2z}\right) = z.
\end{equation}
\end{proposition}

\begin{proof} Since $\mu$ is symmetric, its odd moments are all $0$.  Also, since $\square_\ast \mu = \mu_{|a|^2}$,  the even moments are given by
\[ m_{2k}(\mu) = \|a\|_{2k}^{2k}, \quad k=1,2,\ldots \]
Therefore, 
\[ G_\mu(z) = \frac{1}{z} + \frac{\|a\|_2^2}{z^3} + O\left(\frac{1}{z^5}\right), \quad |z|\to\infty. \]
Hence,
\[ h(s) = -iG_\mu(is) = \frac{1}{s} + \frac{\|a\|_2^2}{s^3}+O\left(\frac{1}{s^5}\right), \quad s\to\infty. \]
It follows that
\[ \frac{1}{h(s)} = s\left(1-\frac{\|a\|_2^2}{s^2} + O\left(\frac{1}{s^4}\right)\right), \quad s\to\infty \]
and
\[ \sqrt{1-4\ll^2h(s)^2} = 1-\frac{2\ll^2}{s^2}+O\left(\frac{1}{s^4}\right). \]
Taking $t=t(s)$ as in the proof of Corollary \ref{cor h -}, we have
\[ \begin{aligned} t(s) &= s - \frac{1}{2h(s)} - \frac{\sqrt{1-4\ll^2h(s)^2}}{2h(s)} \\ &= (\ll^2-\|a\|_2^2)\frac{1}{s}+O\left(\frac{1}{s^3}\right), \end{aligned} \] 
and therefore $t(s)>0$ for $s>s_\ll'$ for some $s_\ll'>0$.  Thus, by Corollary \ref{cor h -},
\[ h_\ll(t(s)) = h(s) \quad \text{for }s>s_\ll' \]
which is precisely the statement of Equation \ref{eq neg root} in the case $z = -ih(s)$.  Since
\[ \lim_{z\to 0}\left[ \mathscr{R}_\mu(z) + \frac{1-\sqrt{1+4\ll^2 z^2}}{2z}\right] = 0, \]
the left-hand-side of Equation \ref{eq neg root} is well-defined and analytic in a complex neighbourhood of $0$.  Hence, as we have shown that the equation holds for all $z$ in an imaginary line segment accumulating at $0$, Equation \ref{eq neg root} holds everywhere in a complex neighbourhood of $0$.
\end{proof}

\subsection{Negative Moments} \label{sect neg moments} From now on, let us normalize $a$ so that $\|a\|_2 = 1$; we therefore only consider $\ll>1$.  We also assume that $a$ is not Haar unitary, which (under this normalization) is equivalent to the requirement that $\|a\|_4 > 1$.  Put
\[ M_\ll = \|a-\ll\|, \quad m_\ll = \|(a-\ll)^{-1}\|^{-1}. \]
Then
\[ \supp(\mu_\ll) \subseteq [-M_\ll,-m_\ll]\cup[m_\ll,M_\ll]. \]
Moreoever, $\pm m_\ll$ and $\pm M_\ll$ are singular points for the Cauchy transform
\[ G_{\mu_\ll}(z) = \int_{\R} \frac{1}{z-x} \mu_\ll(dx), \quad z\in \C-\left([-M_\ll,-m_\ll]\cup[m_\ll,M_\ll]\right). \]
Our goal is to determine the asymptotic behaviour of $m_\ll$ as $\ll\downarrow 1$; we will accomplish this through Equation \ref{eq neg root}.

\medskip

The free cumulants $\kk_n(\mu)$ vanish for odd $n$ since $\mu$ is symmetric, and so $\mathscr{R}_\mu$ is given by the power series
\[ \mathscr{R}_\mu(z) = \kk_2(\mu)z + \kk_4(\mu)z^3 + \kk_6(\mu)z^5 + \cdots \]
for $z$ in a complex neighbourhood of $0$.  Moreover, since $\mu$ is centred, 
\[ \kk_2(\mu) = m_2(\mu) =\|a\|_2^2 = 1 \]
and
\[ \kk_4(\mu) = m_4(\mu) -2 = \|a\|_4^4-2. \]
Define
\begin{equation} \label{eq v(a)} v(a) = \|a\|_4^4 - (\|a\|_2^2)^2 = \kk_4(\mu)+1. \end{equation}
Then $v(a)$ is strictly positive.  We have
\[ \mathscr{R}_\mu(z) = z + (v(a)-1)z^3 + \kk_6(\mu)z^5+ O(z^7). \]
Now, we may expand $\sqrt{1+4\ll^2 z^2}$ as a Taylor series about $0$.  The result is
\[\begin{aligned} \frac{1-\sqrt{1+4\ll^2 z^2}}{2z} &= \frac{1}{2z}\left(1-\sum_{\ell=0}^\infty \binom{1/2}{\ell}(4\ll^2 z^2)^\ell\right) \\
&= -\sum_{\ell=1}^\infty \binom{1/2}{\ell} \ll^{2\ell}(2z)^{2\ell-1} \\
&= \sum_{\ell=1}^\infty (-1)^\ell C_{\ell-1} \ll^{2\ell}(2z)^{2\ell-1} \\
&= -\ll^2 z + \ll^4 z^3 - 2\ll^6 z^5 + \cdots \end{aligned} \]
Here $C_k$ is the {\em Catalan number} $C_k = \frac{1}{k+1}\binom{2k}{k}$.  Now, following Equation \ref{eq neg root}, define
\begin{equation} \label{eq B 1} B_\ll(z) = \mathscr{R}_\mu(z) + \frac{1-\sqrt{1+4\ll^2 z^2}}{2z}. \end{equation}
From the preceding discussion, $B_\ll$ has a the power series expansion
\begin{equation} \label{eq B 2} B_\ll(z) = \left(1-\ll^2\right)z + \left(v(a)-1+\ll^4\right)z^3 + \left(\kk_6(\mu)-2\ll^6\right)z^5 + O(z^7), \end{equation}
in a complex neighbourhood of $0$.

\medskip

By Proposition \ref{prop neg root}, $B_\ll(z)$ and $G_{\mu_\ll}(w)$ are inverse functions of each other when both $|z|$ and $|w|$ are small.  Since $\supp(\mu_\ll)\subseteq\R\setminus(-\d_\ll,\d_\ll)$ with $\d_\ll>0$, we have for $|w|<\d_\ll$,
\begin{equation} \label{eq neg moments 1} \begin{aligned}
G_{\mu_\ll}(w) &= \int_{\R\setminus (-\d_\ll,\d_\ll)} \frac{1}{w-x} \mu_\ll(dx) \\
&= -\int_{\R\setminus (-\d_\ll,\d_\ll)} \frac{1}{x}\left(1+\frac{w}{x}+\frac{w^2}{x^2}+\cdots\right)\,\mu_\ll(dx) \\
&= -\left( m_{-2}(\mu_\ll)\,w + m_{-3}(\mu_\ll)\,w^2 + m_{-4}(\mu_\ll)\,w^3 + \cdots\right),
\end{aligned} \end{equation}
where
\begin{equation} \label{eq neg moments 2}
m_{-k}(\mu_\ll) = \int_{\R} x^{-k} \mu_\ll(dx), \quad k=1,2,\ldots
\end{equation}
are the negative moments of $\mu_\ll$.  Again, since $\mu_\ll$ is symmetric, $m_{-k}(\mu_\ll) = 0$ for $k$ odd.  We will now use the Lagrange inversion formula, together with Equations \ref{eq B 2} and \ref{eq neg moments 1}, to express the even negative moment $m_{-2\ell}(\mu_\ll)$ in terms of the free cumulants $\kk_2(\mu_\ll), \kk_4(\mu_\ll), \ldots, \kk_{2\ell}(\mu_\ll)$ for $\ell\in\N$.  In particular, as can be easily calculated directly,
\begin{equation} \label{eq m -2,-4} m_{-2}(\mu_\ll) = \frac{1}{\ll^2-1}, \quad m_{-4}(\mu_\ll) = \frac{\ll^4-1+v(a)}{(\ll^2-1)^4}. \end{equation}

\begin{lemma} \label{lem Lagrange} Let $v>0$.  The inverse of the function
\[ F(z) = z-vz^3 \]
has the power series expansion
\[ F^{\langle -1 \rangle}(w) = \sum_{k=0}^\infty C^{(2)}_k v^k\,w^{2k+1}, \]
where $C^{(2)}_k = \frac{1}{2k+1}\binom{3k}{k}$.
\end{lemma}

\begin{remark} \label{rk Fuss-Catalan} The numbers $C^{(2)}_k$ in Lemma \ref{lem Lagrange} are the $p=2$ case of the {\em Fuss-Catalan numbers}
\[ C^{(p)}_k = \frac{1}{pk+1}\binom{(p+1)k}{k}. \]
Note that when $p=1$ we recover the standard Catalan numbers.  The appearance of these combinatorially interesting numbers in free probability theory is discussed at length in \cite{Kemp Speicher} and \cite{Larsen}, in addition to more recent papers and preprints \cite{REU}, \cite{Kemp 2}, and \cite{KMRS}.
\end{remark}

\begin{proof} Since $F$ is odd, so is its inverse.  Write the (yet-to-be-determined) coefficients of $F^{\langle -1 \rangle}$ as $F^{\langle -1 \rangle}(w) = w + b_3 w^3 + b_5 w^5 + \cdots$.  Since $F$ is analytic at $0$ and $F'(0)\ne 0$, the Lagrange inversion formula states that
\[ b_{2k+1} = \frac{1}{2k+1}\mathrm{Res}\left(F(z)^{-(2k+1)},0\right). \]
Writing $F(z)^{-(2k+1)} = (1-vz^2)^{-(2k+1)}\,\cdot z^{-2k}\frac{1}{z}$, we see that the residue in question is the coefficient of $z^{2k}$ in the power series expansion of $(1-vz^2)^{-(2k+1)}$, which is equal to
\[ v^k \frac{(2k+1)(2k+2)\cdots (3k)}{k!} = v^k\binom{3k}{k}. \]
This proves the lemma.
\end{proof}

\label{theorem neg moments page}

\begin{theorem} \label{thm neg moments} Let $k$ be a non-negative integer.  Then as $\ll\downarrow 1$,
\begin{equation} \label{eq neg moments}
m_{-2k-2}(\mu_\ll) \sim C^{(2)}_k\frac{v(a)^k}{(\ll^2-1)^{3k+1}}.
\end{equation}
\end{theorem}

\begin{remark} The appearance of the Fuss-Catalan numbers to leading order in Equation \ref{eq neg moments} for the negative moments of $\mu_\ll$ begs for a combinatorial explanation.  Indeed, there is a completely combinatorial proof of Theorem \ref{thm neg moments}; this is the content of Section \ref{sect PSD 1}.  What's more, the lower bound of Theorem \ref{main theorem} can be proved with a simple estimate directly from Theorem \ref{thm neg moments}; this is the content of Section \ref{sect PSD 2}.
\end{remark}

\begin{proof} Put $v = v(a)>0$.  Referring to Equation \ref{eq B 1}, rescale $B_\ll$ and set
\begin{equation} \label{eq F 1}
F_\ll(z) = -(\ll^2-1)^{-3/2} B_\ll\left((\ll^2-1)^{1/2} z\right).
\end{equation}
Then following Equation \ref{eq B 2} and using Equation \ref{eq m -2,-4}, we have
\[ F_\ll(z) = z - (v-1+\ll^4)z^3 - (\ll^2-1)(\kk_6(\mu)-2\ll^6)z^5 - (\ll^2-1)^2(\kk_8+5\ll^6)z^7 + \cdots \]
Hence, the coefficients in the power series for $F_\ll(z)$ converge to the coefficients in the power series
\[ F(z) = z - vz^3 +0 z^5 + 0 z^7 = z-vz^3 \]
as in Lemma \ref{lem Lagrange}.  Therefore, by the continuity of the Lagrange inversion formula, the coefficient $b_{2k+1}^{(\ll)}$ of $z^{2k+1}$ in $F_\ll^{\langle -1 \rangle}(z)$ converges by Lemma \ref{lem Lagrange} to $C^{(2)}_k v^k$ as $\ll\downarrow 1$; that is,
\begin{equation} \label{eq bl limit}
\lim_{\ll\downarrow 1} b^{(\ll)}_{2k+1} = C^{(2)}_k v^k.
\end{equation}
Now inverting Equation \ref{eq F 1} (setting $u=(\ll^2-1)z$),
\[ B_\ll(u) = -(\ll^2-1)^{3/2} F_\ll\left(\frac{u}{(\ll^2-1)^{1/2}}\right), \]
and therefore for $|w|$ small,
\begin{equation} \label{eq G F} \begin{aligned}
G_{\mu_\ll}(w) = B_\ll^{\langle -1 \rangle}(w)
&= (\ll^2-1)^{1/2}\; F_\ll^{\langle -1 \rangle}\left(-\frac{w}{(\ll^2-1)^{3/2}}\right) \\
&= -(\ll^2-1)^{1/2}\; F_\ll^{\langle -1 \rangle}\left(\frac{w}{(\ll^2-1)^{3/2}}\right),
\end{aligned}\end{equation}
the last equality following from the fact that $F_\ll^{\langle -1 \rangle}$ is an odd function.

\medskip

Now, $m_{-2k-2}(\mu_\ll)$ is the coefficient of $w^{2k+1}$ in the power series expansion of $G_{\mu_\ll}(w)$, and so by Equations \ref{eq G F},
\[ m_{-2k-2}(\mu_\ll) = (\ll^2-1)^{1/2}(\ll^2-1)^{-3/2} b^{(\ll)}_{2k+1} \left((\ll^2-1)^{-3/2}\right)^{2k+1} = (\ll^2-1)^{-(3k+1)} b^{(\ll)}_{2k+1}, \]
and hence the Theorem follows from Equation \ref{eq bl limit}.
\end{proof}

\subsection{Proof of Theorem \ref{main theorem}} \label{sect proof of main theorem}

The convergence of the coefficients of $F_\ll$ to those of $F$ as $\ll\downarrow 1$ is not enough to prove our main theorem.  In fact, the convergence is stronger, as we now show.  Choose $\varrho\in(0,\frac{1}{4})$ small enough that $\mathscr{R}_\mu$ is analytic in $B(0,\varrho)$.  Then
\[ B_\ll(z) = \mathscr{R}_\mu(z) + \frac{1-\sqrt{1+4\ll^2 z^2}}{2z} \]
is analytic in $B(0,\varrho)$ for all $\ll\in(1,2)$.  Thence, with $F_\ll$ as in Equation \ref{eq F 1}, $F_\ll$ is analytic and well-defined in $B(0,\varrho/\sqrt{\ll^2-1})$.  Note that $\varrho/\sqrt{\ll^2-1}\to\infty$ as $\ll\downarrow 1$.

\begin{proposition} \label{prop conv compact}
With $F_\ll$ as in Equation \ref{eq F 1} and $F$ as in Lemma \ref{lem Lagrange}, $F_\ll(z)\to F(z)$ uniformly on compact subsets of $\C$ as $\ll\downarrow 1$.
\end{proposition}

\begin{proof} We claim there exists a constant $C>0$ and a $\varrho'>0$ such that, for all $z\in B(0,\varrho')$ and all $\ll\in(1,2)$,
\begin{equation} \label{eq bound}
|B_\ll(z) - (1-\ll^2)z - (v(a)-1+\ll^4)z^3| \le C|z|^5.
\end{equation}
For the moment, assume Equation \ref{eq bound} has been proved.  Then setting $w= (\ll^2-1)^{1/2}z$,
\[ \begin{aligned}  |F_\ll(z)-z&+(v(a)-1+\ll^4)z^3| \\
=& |-(\ll^2-1)^{-3/2} B_\ll\left((\ll^2-1)^{1/2}z\right) - z + (v(a)-1+\ll^4)z^3| \\
=& \left|-(\ll^2-1)^{-3/2} B_\ll(w) - \frac{w}{(\ll^2-1)^{1/2}} + (v(a)-1+\ll^4)\left(\frac{w}{(\ll^2-1)^{1/2}}\right)^3\right| \\
=& (\ll^2-1)^{-3/2} |B_\ll(w) - (1-\ll^2) w - (v(a)-1 +\ll^4) w^3|,
\end{aligned} \]
and by Equation \ref{eq bound} this is
\[ \le (\ll^2-1)^{-3/2}\,C |w|^5 = (\ll^2-1)^{-3/2}\,C |(\ll^2-1)^{1/2} z|^5 = C(\ll^2-1)|z|^5 \]
for $z\in B(0,\varrho'/\sqrt{\ll^2-1})$.  Thus, we have
\[ F_\ll(z)-F(z) = F_\ll(z) - (z-v(a)z^3) = F_\ll(z) - z + v(a)z^3 = F_\ll(z) -z +(v(a)-1+\ll^4)z^3 + (1-\ll^4)z^3, \]
and so
\[ \begin{aligned} |F_\ll(z)-F(z)| &\le |F_\ll(z)-z+(v(a)-1+\ll^4)z^4| + (\ll^4-1)|z|^3 \\
&\le (\ll^4-1)|z|^3 + C(\ll^2-1)|z|^5. \end{aligned} \]
This proves the proposition.  Hence, it remains only to verify the estimate of Equation \ref{eq bound}.  Referring to Equation \ref{eq B 2},
\[ B_\ll(z) - (1-\ll^2)z - (v(a)-1+\ll^4)z^3 = \sum_{\ell=3}^\infty [\kk_{2\ell}(\mu)-(-1)^\ell C_{\ell-1}\ll^{2\ell}]\, z^{2\ell-1}. \]
It is convenient to break this up as a sum of two power series, $B_\ll(z) = D(z) + C_\ll(z)$ where
\[ \begin{aligned} D(z) &= \kk_6(\mu)z^5 + \kk_8(\mu)z^7 + \cdots \\
C_\ll(z) &= -C_2\ll^6 z^5 + C_3 \ll^8 z^7 - \cdots \end{aligned} \]
Now, $D(z)$ is a truncation of the power series for $\mathscr{R}_\mu(z)$, which is convergent and analytic in $B(0,\varrho)$; as $D(z)$ has a $0$ of order $5$ as $0$, it follows that there is a constant $C_1\ge|\kk_6(\mu)|$ such that $|D(z)| \le c_1|z|^5$ in that neighbourhood of $0$.  On the other hand, note that $C_k< 4^k$, and so with $\ll<2$,
\[ |C_\ll(z)| \le 2^{10}|z|^5 + 2^{14}|z|^7 + \cdots = 2^{10}\frac{|z|^5}{1-16|z|^2}, \]
and so choosing $\varrho'<1/4$, we may choose a constant $c_2>0$ with $|C_\ll(z)|\le c_2|z|^5$ for $z\in B(0,\varrho')$.  Setting $C = c_1+c_2$, this proves Equation \ref{eq bound}.
\end{proof}

\pagebreak

\begin{lemma} \label{lem sign variation}
For all $\ll>1$ sufficiently close to $1$, there is a unique $x_\ll\in(0,\frac{1}{\sqrt{v}})$ (where $v = v(a)$) such that the real analytic function $x\mapsto F_\ll'(x)$ has the following sign variation in $[-\frac{1}{\sqrt{v}},\frac{1}{\sqrt{v}}]$.
\begin{itemize}
\item $F_\ll'(x) >0$ for $x\in (-x_\ll,x_\ll)$.
\item $F_\ll'(x) = 0$ for $x = \pm x_\ll$.
\item $F_\ll'(x) <0$ for $x_r < |x| \le \frac{1}{\sqrt{v}}$.
\end{itemize}
Moreover, $x_\ll\to \frac{1}{\sqrt{3v}}$ and $F_\ll(x_\ll) \to \left(\frac{4}{27v}\right)^{1/2}$ as $\ll\downarrow 1$.
\end{lemma}

\begin{proof} The uniform convergence of $F_\ll$ to $F$ on compact subsets of $\C$ implies by standard complex analysis that $F_\ll^{(p)}$ (the $p$th derivative of $F_\ll$) converges uniformly to $F^{(p)}$ on compact subsets of $\C$, for each $p\in\N$.  Since $F^{(3)}(x) = -6v<0$, we can choose $\ll_0>1$ such that for all $\ll\in(1,\ll_0)$,
\begin{equation} \label{eq sign var 1} F_\ll^{(3)}(x) < 0, \quad \text{for} \quad |x|\le \frac{1}{\sqrt{v}}. \end{equation}
Hence, for $\ll\in(1,\ll_0)$, $F_\ll''$ is strictly decreasing on $[-\frac{1}{\sqrt{v}},\frac{1}{\sqrt{v}}]$.  Moreover, $F_\ll''(0)=0$ since $F_\ll$ is an odd function.   Therefore,
\[ F_\ll''(x)<0 \quad \text{for } \quad \textstyle{x\in(0,\frac{1}{\sqrt{v}})}. \]
Hence $F_\ll'$ is strictly decreasing on $[0,\frac{1}{\sqrt{v}}]$ for $\ll\in(1,\ll_0)$.  Moreover,
\[ \begin{aligned}
\lim_{\ll\downarrow 1} F_\ll'(0) &= F'(0) =1 \\
\lim_{\ll\downarrow 1} F_\ll'(\textstyle{\frac{1}{\sqrt{v}}}) &= F'(\textstyle{\frac{1}{\sqrt{v}}}) = -2.
\end{aligned} \]
Therefore we can choose $\ll_1\in(0,\ll_0]$ such that
\[ F_\ll'(0)>0 \quad \text{and} \quad F_\ll'(\textstyle{\frac{1}{\sqrt{v}}}) <0 \quad \text{for} \quad 0<\ll<\ll_1. \]
Hence, for all such $\ll$, the equation $F_\ll'(x) = 0$ has exactly one solution $x_\ll$ in the interval $(0,\frac{1}{\sqrt{v}})$, and $F_\ll'(x) > 0$ for $x\in [0,x_\ll)$, while $F_\ll'(x) <0$ for $x\in (x_\ll, \frac{1}{\sqrt{v}})$.  Since $F_\ll'$ is an even function, the above-stated sign variation holds for all $\ll\in(1,\ll_1)$.

\medskip

Now, note that $\frac{1}{\sqrt{3v}}$ is a critical point for $F$, and also for $\e\in(0,\frac{1}{2})$,
\[ F'(\textstyle{\frac{1-\e}{\sqrt{3v}}}) >0 \quad \text{and} \quad F'(\textstyle{\frac{1+\e}{\sqrt{3v}}}) <0. \]
Therefore, by the uniform convergence,
\[ F_\ll'(\textstyle{\frac{1-\e}{\sqrt{3v}}}) >0 \quad \text{and} \quad F_\ll'(\textstyle{\frac{1+\e}{\sqrt{3v}}}) <0, \]
for $\ll>1$ sufficiently close to $1$.  It follows that $\frac{1-\e}{\sqrt{3v}} < x_\ll < \frac{1+\e}{\sqrt{3v}}$ eventually as $\ll\downarrow 1$.  This shows that $\lim_{\ll\downarrow 1} x_\ll = \frac{1}{\sqrt{3v}}$, and by the uniform convergence of $F_\ll$ to $F$ on $[-v,v]$, it follows that
\[ \lim_{\ll\downarrow 1} F_\ll(x_\ll) = F(\textstyle{\frac{1}{\sqrt{3v}}}) = \left(\frac{4}{27v}\right)^{1/2}, \]
as required.
\end{proof}

This finally brings us to the proof of the main theorem.
\begin{proof}[Proof of Theorem \ref{main theorem}] It follows from Lemma \ref{lem sign variation} that for all $\ll>1$ sufficiently close to $1$, $F_\ll^{\langle -1 \rangle}$ has an analytic extension to a complex open neighbourhood of the interval $(-F_\ll(x_\ll),F_\ll(x_\ll))$, but that both endpoints of the interval are singular points of $F_\ll^{\langle -1 \rangle}$.  Since
\[ G_{\mu_\ll}(z) = -(\ll^2-1)^{1/2} F_\ll^{\langle -1 \rangle}\left(\frac{z}{(\ll^2-1)^{3/2}}\right) \]
for $|z|$ small, it follows that $G_{\mu_\ll}$ has an analytic extension to a complex neighbourhood of
\[ I_\ll = \left(-(\ll^2-1)^{3/2}F_\ll(x_\ll),(\ll^2-1)^{3/2}F_\ll(x_\ll)\right), \]
but the endpoints of the interval are singular points for $G_{\mu_\ll}$.  Therefore,
\[ \pm(\ll^2-1)^{3/2} F_\ll(x_\ll) \in \supp(\mu_\ll), \]
while
\[ I_\ll\cap \supp(\mu_\ll) = \emptyset. \]
Since $\mu_\ll = \tilde{\mu}_{|a-\ll|}$, it follows that
\[ \|(a-\ll)^{-1}\| = \frac{1}{(\ll^2-1)^{3/2}F_\ll(x_\ll)}, \]
for $\ll>1$ sufficiently close to $1$.  As $\ll\downarrow 1$, this tends (by Lemma \ref{lem sign variation}) to
\[ \frac{1}{2^{3/2}}\,\frac{1}{(\ll-1)^{3/2}}\, \left(\frac{27v}{4}\right)^{1/2}, \]
thus proving Theorem \ref{main theorem}.
\end{proof}

\section{The Combinatorics of Negative Moments} \label{sect Neg Moments}

In this final section, we provide a new combinatorial framework for even moments of the absolute resolvent of an $\mathscr{R}$-diagonal operator (that is, negative moments of $|\ll-a|^2$ where $a$ is $\mathscr{R}$-diagonal and $\ll>\|a\|_2$).  This approach, through {\em partition structure diagrams}, is used in Section \ref{sect PSD 1} below to give a new proof of (a more refined statement of) Theorem \ref{thm neg moments}.  In Section \ref{sect PSD 2}, we show how knowledge of these moments alone yields the sharp lower bound of Theorem \ref{main theorem}.

\subsection{Partition Structure Diagrams} \label{sect PSD 1}

Let us normalize $a$ once again so that $\|a\|_2 = 1$, and let $\ll>1$.  For convenience, let $r = 1/\ll \in (0,1)$.  Then we may rewrite $R_a(\ll) = (\ll-a)^{-1} = r(1-ra)^{-1}$.  Hence, for any positive integer $k$,
\[ \begin{aligned} |R_a(\ll)|^{2(k+1)} &= [(\ll-a)(\ll-a^\ast)]^{-(k+1)} = \left[(\ll-a^\ast)^{-1}(\ll-a)^{-1}\right]^{k+1} \\
&= r^{2(k+1)} \left[(1-ra^\ast)^{-1}(1-ra)^{-1}\right]^{k+1}. \end{aligned} \]
Expanding the geometric series inside this term gives
\[ |R_a(\ll)|^{2(k+1)} = r^{2(k+1)} \left[\sum_{n\ge 0} r^n a^{\ast n}\,\sum_{m\ge 0} r^m a^m\right]^{k+1}. \]
Expanding this product of summations we have
\[ |R_a(\ll)|^{2(k+1)} = r^{2(k+1)} \sum_{n_0,\ldots,n_k\atop m_0,\ldots,m_k} r^{n_0+\cdots+n_k+m_0+\cdots+m_k} a^{\ast n_0} a^{m_0} \cdots a^{\ast n_k} a^{m_k}. \]
Since $a$ is $\mathscr{R}$-diagonal, it is rotationally-invariant, and so only monomials with equal numbers of $a$ and $a^\ast$ can have non-zero mean in the state $\phi$.  Thus,
\begin{equation} \label{eq neg mom expansion 1}
m_{-2k-2}(\mu_\ll) = \phi\left(|R_a(\ll)|^{2(k+1)}\right) = r^{2(k+1)}\hspace{-0.3in}\sum_{{n_0,\ldots,n_k\atop m_0,\ldots,m_k}\atop n_0+\cdots +n_k = m_0+\cdots+m_k} \hspace{-0.3in} r^{2(n_0+\cdots+n_k)}\, \phi(a^{\ast n_0} a^{m_0} \cdots a^{\ast n_k} a^{m_k}).
\end{equation}
We now employ Equation \ref{eq R-diag moment cumulant} expanding the mean as
\[ \phi(a^{\ast n_0} a^{m_0} \cdots a^{\ast n_k} a^{m_k}) = \sum_{\pi\in NC(n_0,m_0,\ldots,n_k,m_k)} \kk_\pi[a^{\ast,n_0},a^{,m_0},\ldots,a^{\ast,n_k},a^{,m_k}]. \]

\begin{notation} The following shorthand notations will make the text below far more readable.
\begin{itemize}
\item Let $\mx{n},\mx{m}$ stand for multi-indices $(n_0,\ldots,n_k)$ and $(m_0,\ldots,m_k)$.
\item The sum of a multi-index $n_0+\cdots+n_k$ is denoted $|\mx{n}|$.
\item Denote the interleaved multi-index $(n_0,m_0,\ldots,n_k,m_k)$ as $\mx{n}\merge\mx{m}$.
\item Denote the cumulant $\kk_{2\ell}[a,a^\ast,\ldots,a,a^\ast]$ as $\alpha_\ell(a)$; note, by traciality, $\alpha_\ell(a)$ is also equal to the cumulant $\kk_{2\ell}[a^\ast,a,\ldots,a^\ast,a]$.
\end{itemize}
\end{notation}

Because $a$ is $\mathscr{R}$-diagonal, the mixed cumulant $\kk_\pi[a^{\ast,n_0},a^{,m_0},\ldots,a^{\ast,n_k},a^{,m_k}]$ is a product of terms $\alpha_\ell(a)$ with $\ell$ between $1$ and $k+1$.  The precise product is determined by the {\em block profile} of the partition $\pi$.

\begin{definition}
Let $\mx{n},\mx{m}$ be multi-indices, and let $\pi\in NC(\mx{n}\merge\mx{m})$.  The {\bf block profile} $\mx{pr}(\pi)$ is the multi-index $\mx{p} = (p_1,\ldots,p_{k+1})$, where $p_1$ is the number of blocks in $\pi$ of size $2$, $p_2$ is the number of blocks in $\pi$ of size $4$, and so forth.  Note that, in this case, the sum $|\mx{n}|+|\mx{m}| = 2|\mx{n}|$ is equal to $2p_1+4p_2+\cdots+2(k+1)p_{k+1}$.  Denote this number as $\ex(\mx{p})$,
\[ \ex(\mx{p}) = 2\sum_{j=1}^{k+1} j\,p_j. \]
\end{definition}

By definition, if $\mx{pr}(\pi) = \mx{p}$, then the mixed cumulant $\kk_\pi[a^{\ast,n_0},a^{,m_0},\ldots,a^{\ast,n_k},a^{,m_k}]$ is equal to $\alpha_1(a)^{p_1}\cdots \alpha_{k+1}(a)^{p_{k+1}}$.  Denote this product as $\alpha(a)^{\mx{p}}$.  We can thus re-index the sum in Equation \ref{eq neg mom expansion 1} as follows:
\begin{equation} \label{eq block profile expansion} m_{-2k-2}(\mu_\ll) = r^{2(k+1)} \sum_{\mx{n},\mx{m}\atop|\mx{n}| =|\mx{m}|} \sum_{\pi\in NC(\mx{n}\merge\mx{m})} \alpha(a)^{\mx{pr}(\pi)}\,r^{\ex(\mx{pr}(\pi))}. \end{equation}

The idea now is to reindex the sum over $\mx{n},\mx{m}$ and $\pi$ in terms of a new set of combinatorial objects: {\em partition structure diagrams}.  Set $V_{k+1} = \{(1,1),(1,\ast),(2,1),(2,\ast),\ldots,(k+1,1),(k+1,\ast)\}$; view $V_{k+1}$ as vertices (in sequence) on the boundary of a disc.  It is possible to encode all the information in a pair $\mx{n},\mx{m}$ and a partition $\pi\in NC(\mx{n}\merge\mx{m})$ succinctly in terms of $V_{k+1}$, as follows.  Consider subsets of $V_{k+1}$ of the form $P=\{(v_1,1),(w_1,\ast),\ldots,(v_\ell,1),(w_\ell,\ast)\}$ where $v_1\le w_1<v_2\le w_2 < \cdots < v_{\ell}\le w_{\ell}$ and $\ell\le k+1$; viewed on the disc, $P$ is a  convex $2\ell$-gon whose vertices alternate between $1$ and $\ast$.  (Note: $\ell=1$ is allowed -- a $2$-gon is a line-segment.)

\medskip

For any $\mx{n},\mx{m},\pi\in NC(\mx{n}\merge\mx{m})$, assign non-negative integers to the polygonal subset $P$ of $V_{k+1}$ as follows. $P$ is assigned the number of $2\ell$-blocks in $\pi$ that connect vertices in sequence as follows: a vertex from the $v_1$ run of $1$s, then a vertex from the $w_1$ run of $\ast$s, then a vertex from the $v_2$ run of $1$s, and so forth.  As such, a triple $\mx{n},\mx{m},\pi$ yields an {\em labeled polygonal diagram}, or LPD.  Figure demonstrates this procedure.

\begin{figure}[htbp]
\begin{center}
\input{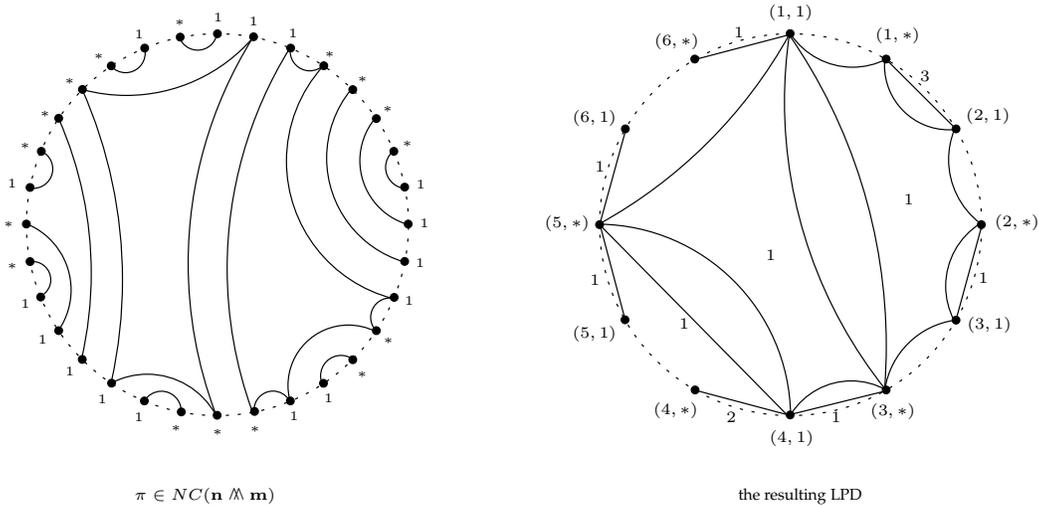}
\caption{\label{fig LPD} A partition in $\pi\in NC(\mx{n}\merge\mx{m})$, where $\mx{n} = (3,4,2,5,1,1)$ and $\mx{m} = (4,2,3,2,4,1)$, together with the resulting labeled polygonal diagram.  $2$-gons are drawn as straight line-segments, while non-degenerate polygons are drawn non-convex so all lines and labels are clearly visible.  Those polygonal subsets not appearing in the LPD have label $0$.}
\end{center}
\end{figure}

\noindent Note, in Figure \ref{fig LPD}, those polygons with more than $2$ sides have label $1$ or $0$.  This is a general phenomenon, as the reader can easily check: since $\pi$ is non-crossing, there can be at most one block of size $2\ell$ ($\ell>1$) connecting a set of vertices in the same $1$- and $\ast$-runs.  However, this restriction does not apply to $2$-blocks: there can be many nested pairings, as seen in Figure \ref{fig LPD}.  There is a simple geometric explanation here: for fixed $\mx{n},\mx{m}$, the map from $\pi$ to the LPD (viewed as a collection of inscribed polygons in a disc) is a compression.  Two $(2\ell)$-blocks ($\ell>1$) in $\pi$ with the same image under the compression would have non-trivially intersecting interiors, and therefore would cross.  Since $2$-blocks have no interior, on the other hand, they can compress to the same $2$-gon without crossing.

\medskip

Not every LPD is the compression of a partition: as explained, to come from a partition, the label of any non-degenerate polygon must be $0$ or $1$.  But there are further restrictions.  For example, if the label of the $2$-gon joining $(2,1)$ to $(4,\ast)$ in Figure \ref{fig LPD} were non-zero, no non-crossing pairing could compress to that LPD.  The restriction, of course, is that the polygons with non-zero label in the LPD should be non-crossing as well.  This brings us to the notion of a(n unlabled) {\em partition structure diagram}.

\begin{definition} \label{def PSD}
A {\bf partition structure diagram} or {\bf PSD} with $2(k+1)$ vertices, $D$, is a collection of polygons, each with an even number of sides, inscribed in the disc $V_{k+1}$, with the following additional properties:
\begin{itemize}
\item The vertices of any polygon $P\in D$ are among the vertices $V_{k+1}$, and any edge in $P$ connects a $1$ with a $\ast$.
\item The intersection of any two polygons in $P$ has $0$ area; that is, if any two intersect, it is along a common set of edges or vertices only.
\end{itemize}
Denote the set of all such $D$ by $PSD_{k+1}$.
\end{definition}
Following our discussion, if $D=(P_1,\ldots, P_s)$ is a PSD, and we label it with positive integers $L(P_1),\ldots,L(P_s)$, then it can only be the compression of a partition in some $NC(\mx{n},\mx{m})$ if the label of any non-degenerate polygon in $D$ is $1$.  Call such a labeling $L$ {\em valid}. 

\medskip

In fact, it is easy to see that any such labeled PSD is the compression of a {\em unique} partition.  Here is the algorithm: begin by decompressing each labeled $2$-gon into the requisite number of nested $2$-blocks.  Then, to avoid crossings, each remaining non-degenerate polygon can be inserted in one and only one way -- closer to the center of the disc than any surrounding $2$-blocks.

\medskip

Hence, there is a bijection between the set $\{(\mx{n},\mx{m},\pi)\,;\,\mx{n},\mx{m}\in\N^{k+1},|\mx{n}|=|\mx{m}|,\pi\in NC(\mx{n}\merge\mx{m})\}$ and the set $\{(D,L)\,;\,D\in PSD_{k+1}, L\text{ is a valid labeling of  }D\}$.  What's more, this bijection preserves the required statistic $\mx{pr}(\pi)$ in a recordable way: if $D\in PSD_{k+1}$ with component polygons $P_1,\ldots,P_s$, and if $L$ is a valid labeling of $D$, then the number of $2\ell$-blocks in the unique $\pi$ whose compression is $(D,L)$ is equal to $\sum_{P\in D} L(P)\1\{|P|=2\ell\}$ (here $|P|$ denotes the number of sides of $P$).  Hence, we can refer to the profile $\mx{pr}(D,L)$.  Note, in particular, that for this $\pi$,
\[ \ex(\mx{pr}(\pi)) = \sum_{P\in D} L(P)|P|. \]
We will refer to this sum simply as $\ex(D,L)$.

\medskip

We can thence re-index the summation of Equation \ref{eq block profile expansion} as follows.

\begin{equation} \label{eq PSD expansion}
m_{-2k-2}(\mu_\ll) = r^{2(k+1)} \sum_{D\in PSD_{k+1}} \sum_{L\text{ labels }D} \alpha(a)^{\mx{pr}(D,L)}\,r^{\ex(D,L)}.
\end{equation}
The sum in Equation \ref{eq PSD expansion} is quite complicated, but fortunately we are only looking for the leading order term in $\frac{1}{\ll^2-1} = \frac{r^2}{1-r^2}$.  To achieve this, it is convenient to break up the sum into two parts: over those $D$ with polygons having no more than $4$ sides, and the remaining $D$ that contain a polygon with at least $6$ sides.  Denote these two sets as $PSD^{\le 4}_{k+1}$ and $PSD^6_{k+1}$.

\medskip

For the sum over $PSD^{\le 4}_{k+1}$, we break up the sum according to the profile: look at those $D$ containing $s$ $2$-gons $P_1,\ldots,P_s$ and $t$ $4$-gons $Q_1,\ldots,Q_t$.  In this case, from our previous discussion, any valid labeling gives label $1$ to each of $Q_1,\ldots,Q_t$ (the label must be $\le 1$, and since we suppose each is present, the labels must be $1$).  On the other hand, valid labels for $P_1,\ldots,P_s$ range independently among the positive integers (again, $0$ is excluded since we suppose all $s$ are present).  For such a labeling $L$, we have
\begin{equation} \label{eq v(D,L)} \ex(D,L) = 2(L(P_1)+\cdots+L(P_s)) + 4(L(Q_1)+\cdots+L(Q_t)) = 2(L(P_1)+\cdots+L(P_s)) + 4t. \end{equation}
For the cumulant term, the profile $\mx{pr}(D,L)$ has $L(P_1)+\cdots+L(P_s)$ $2$-blocks, and $L(Q_1)+\cdots+L(Q_t) = t$ $4$-blocks, and so
\[ \alpha(a)^{\mx{pr}(D,L)} = \alpha_1(a)^{L(P_1)+\cdots+L(P_s)}\cdot\alpha_2(a)^t. \]
However, $\alpha_1(a) = \kk_2[a,a^\ast] = \|a\|_2^2 = 1$ by our normalization, and so so the cumulant term simply becomes
\begin{equation} \label{eq cumulant term}
\alpha(a)^{\mx{pr}(D,L)} = \alpha_2(a)^t.
\end{equation}
So, letting $\eta_1,\ldots,\eta_s$ denote the labels $L(P_1),\ldots,L(P_s)\in\N-\{0\}$, we can express the sum over $PSD^{\le 4}_{k+1}$ as
\begin{equation} \label{eq PSD4} \begin{aligned}
r^{2(k+1)}\sum_{D\in PSD_{k+1}^{\le 4}} &\sum_{L\text{ labels }D} \alpha(a)^{\mx{pr}(D,L)}\,r^{\ex(D,L)} \\
&= r^{2(k+1)}\sum_{s,t=0}^\infty \Pi_{k+1}(s,t)\sum_{\eta_1,\ldots,\eta_s=1}^\infty \alpha_2(a)^t\, r^{2(\eta_1+\cdots+\eta_s)+4t}, \end{aligned} \end{equation}
where $\Pi_{k+1}(s,t)=\#\{D\in PSD_{k+1}^{\le 4}\,;\,D\text{ has $s$ $2$-gons $\&$ $t$ $4$-gons}\}$.  The internal sum (over $\eta_1,\ldots,\eta_s$) factors as a product of $s$ independent summations,
\[ \sum_{\eta_1,\ldots,\eta_s=1}^\infty r^{2(\eta_1+\cdots+\eta_s)+4t} 
= r^{4t}\left(\sum_{\eta=1}^\infty r^{2\eta}\right)^s = r^{4t}\left(\frac{1}{1-r^2}-1\right)^s. \]
So, summing over $PSD^{\le 4}_{k+1}$ yields
\begin{equation} \label{eq PSD4 2} r^{2(k+1)}\sum_{s,t=0}^\infty \Pi_{k+1}(s,t)\left(\alpha_2(a)r^4\right)^t\,\left(\frac{1}{1-r^2}-1\right)^s. \end{equation}
Of course, the indices $s,t$ really have finite ranges: the $2$-gons are chosen from among $\binom{k+1}{2}$ and the $4$-gons from $\binom{k+1}{4}$ possible configurations, meaning that the constant $\Pi_{k+1}(s,t)$ is $0$ for large enough $s,t$.  Since we seek the highest-order term in $\frac{1}{\ll^2-1}$, we are interested in the largest $s$ for which $\Pi_{k+1}(s,t)\ne 0$ for any fixed $t$: we expand the binomial to the power $s$ in Equation \ref{eq PSD4 2}, and are interested only in the term $\left(\frac{1}{1-r^2}\right)^s = \frac{\ll^{2s}}{(\ll^2-1)^s}$ of highest order.

\medskip

The key observation here is that, for any diagram $D$ with only $2$- and $4$-gons, additional $2$-gons may be added without crossings until the skeleton of $2$-gons partitions the area of $V_{k+1}$ into $4$-gons -- i.e.\ it produces a tiling of $V_{k+1}$ by $4$-gons.    There may be many distinct $4$-gon tilings that can result from such a completion.  Nevertheless, this means that, to enumerate those $D$ with $s$ $2$-gons and $t$ $4$-gons, we may begin by considering any possible $4$-gon tiling of $V_{k+1}$, and then consider all possible ways of including $s$ lines and $t$ $4$-gons in it.  This construction is purely combinatorial and the details are left to the interested reader.  The procedure is exemplified in Figure \ref{fig completion to tiling}.

\medskip

\begin{figure}[htbp]
\begin{center}
\input{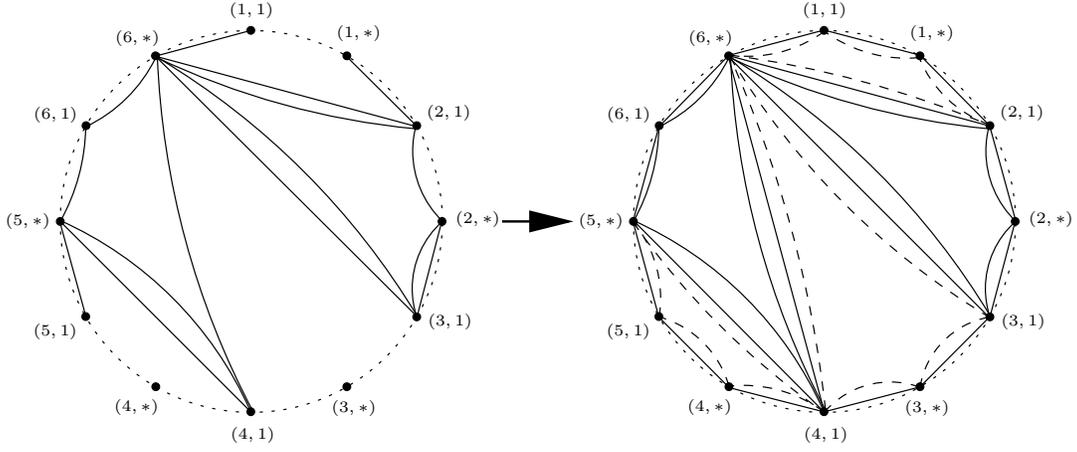}
\caption{\label{fig completion to tiling} A diagram in $PSD^{\le 4}_6$ with $7$ $2$-gons (shown as straight lines) and $2$ $4$-gons.  An additional $9$ $2$-gons can be added to produce a tiling of the $12$-gon by $4$-gons.}
\end{center}
\end{figure}

The maximal $s$, for given $t$, for which $\Pi_{k+1}(s,t)\ne 0$, is therefore given by the number of line-segments in a $4$-gon tiling of $V_{k+1}$.  The following classical results may be found in \cite{4-gon}.

\begin{lemma} \label{lem 4-gon tiling}
The number of line-segments in any $4$-gon tiling of a $(2k+2)$-gon is $3k+1$ (including the boundary edges).  The number of such distinct tilings is given by the Fuss-Catalan number $C^{(2)}_k$.
\end{lemma}

\begin{remark} It is well-known that the Catalan number $C_k = C^{(1)}_k$ counts the number of $3$-gon (triangular) tilings of a $(k+2)$-gon.  This is the $p=1$ case of the following theorem, proved in \cite{4-gon}: The number of $(p+2)$-gon tilings of a $(pk+2)$-gon is the Fuss-Catalan number $C^{(p)}_k$.   The second statement of Lemma \ref{lem 4-gon tiling} is the $p=2$ case of this theorem.
\end{remark}

Thus, the largest $s$ for which $\Pi_{k+1}(s,t)\ne 0$ is $s=3k+1$, provided $t$ is not so large that there cannot be $t$ $4$-gons inserted into the tiling provided by the $3k+1$ $2$-gons.  It is an easy matter to count that there are $k$ distinct $4$-gons in any $4$-gon tiling of $V_{k+1}$, and so we have $\Pi_{k+1}(s,t)>0$ whenever $0\le s\le 3k+1$ and $0\le t\le k$.  For fixed $t$ in this range, there are $\binom{k}{t}$ distinct choices of positions for the $t$ $4$-gons out of the $k$ slots.  Hence, we have proved the following:
\begin{equation} \label{eq binomial coeff} \Pi_{k+1}(3k+1,t) = \binom{k}{t}C^{(2)}_k, \quad 0\le k\le t. \end{equation}
And, of course, $\Pi_{k+1}(s,t)=0$ for $t>k$.  Combining Equations \ref{eq PSD4 2} and \ref{eq binomial coeff}, we have that the leading-order coefficient in $\frac{1}{1-r^2}$ is contained in the expansion of
\[ \begin{aligned}  & r^{2(k+1)} \sum_{t=0}^k  \Pi_{k+1}(3k+1,t) \left(\alpha_2(a)r^4\right)^t\,\left(\frac{1}{1-r^2}-1\right)^{3k+1} \\
=\;& r^{2(k+1)}C^{(2)}_k \left(\frac{1}{1-r^2}-1\right)^{3k+1}\,\sum_{t=0}^k \binom{k}{t}\,\left(\alpha_2(a)r^4\right)^t.
\end{aligned} \]
The sum over $t$ simplifies, via the binomial theorem, to $(1+\alpha_2(a)r^4)^k$.  On the other hand, if we expand the binomial to the power $3k+1$, and reserve only the highest order term $\left(\frac{1}{1-r^2}\right)^{3k+1}$, we find that the leading order contribution to the $PSD^{\le 4}_{k+1}$ sum is given by
\begin{equation} \label{eq leading order intermediate 1} C^{(2)}_k\left(1+\alpha_2(a)r^4\right)^k\frac{r^{2(k+1)}}{(1-r^2)^{3k+1}}. \end{equation}
At this point, it is useful to manipulate the expression to bring it to a more familiar form.  If we return to the variable $\ll = 1/r$, Equation \ref{eq leading order intermediate 1} becomes
\[ C^{(2)}_k \left(1+\alpha_2(a)\ll^{-4}\right)^k\frac{\ll^{-2(k+1)}}{(1-\ll^{-2})^{3k+1}} \]
which equals
\[ C^{(2)}_k\left(\ll^4+\alpha_2(a)\right)^k\ll^{-4k} \frac{\ll^{-2k-2}\ll^{62+2}}{(\ll^2-1)^{3k+1}}. \]
Simplifying the last expression, and substituting the value $\alpha_2(a) = \|a\|_4^4-2 = v(a)-1$, we have finally that the leading order term in the $PSD^{\le 4}_{k+1}$ expansion is
\begin{equation} \label{eq leading order intermediate 2}
C^{(2)}_k\frac{(\ll^4-1+v(a))^k}{(\ll^2-1)^{3k+1}}.
\end{equation}
(Equation \ref{eq leading order intermediate 2} should be compared, favourably, with Equation \ref{eq m -2,-4}.)  It is worth noting that, since $\ll^4-1$ is divisible by $\ll^2-1$, we could again expand Equation \ref{eq leading order intermediate 2} as a binomial, and the leading order term in the highest power $\left(\frac{1}{\ll^2-1}\right)^{3k+1}$ is simply $C^{(2)}_k v(a)^k$.  This is the desired entire leading order term according to Theorem \ref{thm neg moments}, and so we must now argue away all the terms in the $PSD^6_{k+1}$ expansion.  That is, summing up and returning to Equation \ref{eq PSD expansion}, we have

\begin{equation} \label{eq leading order 1} \begin{aligned}
m_{-2k-2}(\mu_\ll) =  C^{(2)}_k&\frac{(\ll^4-1+v(a))^k}{(\ll^2-1)^{3k+1}} \\
&+ \ll^{-2(k+1)} \sum_{D\in PSD^6_{k+1}} \sum_{L\text{ labels }D} \alpha(a)^{\mx{pr}(D,L)}\,\ll^{-\ex(D,L)} \\
&+ \text{lower order terms in }\frac{1}{\ll^2-1}.
\end{aligned}\end{equation}

\noindent We now must proceed to show that the middle terms in Equation \ref{eq leading order 1} are all sub-leading order.

\begin{remark} Consider the special case $a=c$ is circular.  Here $\alpha_\ell(a)=0$ for $\ell>1$, and so the only terms in Equation \ref{eq PSD expansion} that contribute come from diagrams with {\em only $2$-gons}.  Referring to Equation \ref{eq PSD4 2}, in this case we have the exact formula
\[ m_{-2k-2}(\mu_\ll) = \ll^{-2(k+1)}\sum_{s=0}^{3k+1} \Pi_{k+1}(s,0)\frac{1}{(\ll^2-1)^s}. \]
The leading coefficient (with $s=3k+1$) is $C^{(2)}_k$; the lower-order coefficients are very challenging to calculate.  Nevertheless, the fact remains that the moment is a polynomial in $\frac{1}{\ll^2-1}$, which is a point of independent interest.
\end{remark}

Now let us consider the $PSD^6_{k+2}$ terms in Equation \ref{eq leading order 1}.  In fact, we can give a general expansion like the one given above for any allowed profile of PSD $D$.  Consider those $D$ with $s_1$ $2$-gons, $s_2$ $4$-gons, and so on through $s_{k+1}$ $2(k+1)$-gons.  (The condition that $D$ is in $PSD^6_{k+1}$ means there is some $\ell\ge 3$ with $s_\ell>0$.) Denote the $2\ell$-gons  present as $P^{\ell}_{1},\ldots,P^{\ell}_{s_\ell}$.  Valid labeling of such polygons again takes the following form: each of $P^1_1$ through $P^1_{s_1}$ can be labeled with any positive integer, while those $P^{\ell}_j$ with $\ell>1$ must have label $1$.  The general expansion is then
\[ r^{2(k+1)}\sum_{s_1,\ldots,s_{k+1}} \Pi_{k+1}(s_1,\ldots,s_{k+1}) \sum_{\eta_1,\ldots,\eta_{s_1}=1}^\infty \alpha(a)^{\mx{pr}}\, r^\ex, \]
where $\Pi_{k+1}(s_1,\ldots,s_{k+1})$ is the number of $D\in PSD_{k+1}$ with $s_1$ $2$-gons, $s_2$ $4$-gons, and so on through $s_{k+1}$ $2(k+1)$-gons, and in this general setting we have
\[ \alpha(a)^{\mx{pr}} = \alpha_2(a)^{s_2}\cdots \alpha_{k+1}(a)^{s_{k+1}}, \quad \ex = 2(\eta_1+\cdots+\eta_{s_1}) + 4s_2+\cdots +2(k+1)s_{k+1}. \]
(The expression for $\alpha(a)^{\mx{pr}}$ should also contain the term $\alpha_1(a)^{\eta_1+\cdots+\eta_{s_1}}$, but as above the normalization $\|a\|_2=1$ sets this term equal to $1$.)  Again the internal sum simplifies to a product, and we have
\begin{equation} \label{eq full expansion} m_{-2k-2}(\mu_\ll) = r^{2(k+1)}\sum_{s_1,\ldots,s_{k+1}} \Pi_{k+1}(s_1,\ldots,s_{k+1})\, r^{4s_2+6s_3+\cdots+2(k+1)s_{k+1}}\,\left(\frac{1}{1-r^2}-1\right)^{s_1}. \end{equation}
(This expression is completely general and we could have started with it instead of considering the case that $s_\ell=0$ for $\ell>2$ as we did.  Also, in the preceding notation, the terms $\Pi_{k+1}(s,t)$ would now be denoted $\Pi_{k+1}(s,t,0,\ldots,0)$.)  The range of each $s_j$ is through a finite set, and so again we see it is only the $2$-gons that yield an infinite expansion -- all other terms are polynomial in $r^2 = 1/\ll^2$.  Now, consider the portion of this sum corresponding to $PSD_{k+1}^6$: those terms for which at least one $s_\ell$ with $\ell>2$ is non-zero.  The following lemma shows that the leading order in $\frac{1}{\ll^2-1}$ cannot be achieved in this case.

\begin{lemma}
Let $2<\ell\le k+1$, and suppose that $s_\ell>0$. Then $\Pi_{k+1}(s,s_2,\ldots,s_{k+1}) = 0$ for $s\ge 3k+1$.
\end{lemma}

\noindent In other words, there are no $\left(\frac{1}{\ll^2-1}\right)^{3k+1}$ or higher-order terms in the $PSD^6_{k+1}$ expansion.  The idea behind the proof is quite easy: the maximal number of non-crossing $2$-gons in $V_{k+1}$ is $3k+1$, which is achieved by any $4$-gon tiling.  According to Lemma \ref{lem 4-gon tiling}, any $2\ell$-gon $P^{\ell}_j$ in $D$ with $\ell>2$ can be subdivided into $\ell-1$ $4$-gons (in $C^{(2)}_{k-1}$ distinct ways) by adding $\ell-2$ lines.  The modified $D$ can have no more than the maximal number, $3k+1$, line segments, which means that the original $D$ can have no more than $(3k+1)-(\ell-2)$ $2$-gons; if $\ell\ge 3$, this means the maximal order is not achieved.  This is demonstrated in Figure \ref{fig dividing a 6-gon}.

\begin{figure}[htbp]
\begin{center}
\input{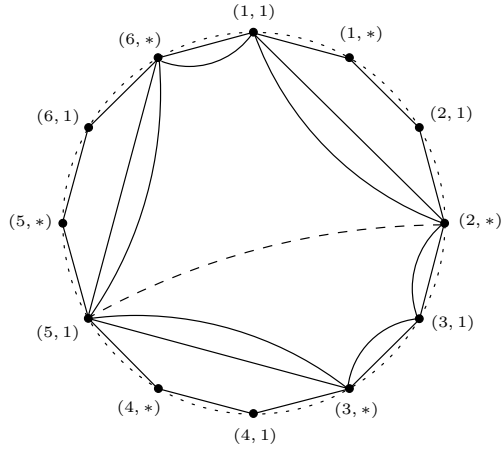}
\caption{\label{fig dividing a 6-gon}A PSD $D$ with a $6$-gon.  This $6$-gon can be subdivided into two $4$-gons (demonstrated by the dotted line), showing that $D$ cannot have the maximal possible number ($16$) of $2$-gons.}
\end{center}
\end{figure}

Thus, all of the $PSD_{k+1}^6$ terms in Equation \ref{eq leading order 1} contribute to {\em sub}-leading order, and referring to Equation \ref{eq leading order 1}, we have thus proved that

\begin{equation} \label{eq leading order 2}
m_{-2k-2}(\mu_\ll) =  C^{(2)}_k\frac{(\ll^4-1+v(a))^k}{(\ll^2-1)^{3k+1}} + \text{lower order terms in }\frac{1}{\ll^2-1}.
\end{equation}

\noindent As explained above, we may safely ignore the $\ll^4-1$ term inside the numerator as it cancels to yield lower-order contributions.  As such, Equation \ref{eq leading order 2} provides our combinatorial proof of Theorem \ref{thm neg moments}.  In fact, Equation \ref{eq leading order 2} is a finer statement, whose character warrants further study.  On the other hand, the combinatorial approach shows that all the expressions involved are in fact polynomials.  We record this as the following theorem, which is evident from Equation \ref{eq full expansion} noting that $\frac{1}{1-r^2}-1 = \frac{1}{\ll^2-1}$.

\label{thm moment polynomial page}

\begin{theorem} \label{thm moment polynomial}  Let $a$ be $\mathscr{R}$-diagonal in a $\mathrm{II}_1$-factor with trace $\phi$, normalized so that $\|a\|_2 = 1$, and let $k\ge 0$.  Then there is a {\em polynomial} $P_{k+1}^a$ in two variables so that
\[ \phi[(\ll-a)(\ll-a)^\ast]^{-(k+1)} = P_{k+1}^a\left(\frac{1}{\ll^2-1},\frac{1}{\ll^2}\right). \]
for $\ll>1$. \end{theorem}

\begin{remark} \label{rk polynomial funny business} It should be noted that, following Equation \ref{eq full expansion} and referring to Equation \ref{eq leading order intermediate 1}, the polynomial $P^{a}_{k+1}(x,y)$ can be taken so that its leading term in $x$ is
\[ C^{(2)}_k y^{2(k+1)}\left(1+(v(a)-1)y^2\right)^k x^{3k+1}. \]
This result looks quite different from the expression we would expect from Equation \ref{eq leading order 2}.  The point here is that the two variables $x,y$ of $P^a_{k+1}$ are not truly independent, since the instantiations $x=\frac{1}{\ll^2-1}$ and $y=\frac{1}{\ll^2}$ lead to the relation $x-y = xy$.  Using this, it is possible to transform the above expression into many other forms, and so it is impossible to speak of {\em the} polynomial in Theorem \ref{thm moment polynomial}.  \end{remark}

\subsection{Theorem \ref{main theorem} via negative moments} \label{sect PSD 2} The asymptotic upper-bound of Theorem \ref{main theorem}, with a non-sharp constant, is actually non-asymptotic.  This is an easy application of the strong Haagerup inequality in \cite{Kemp Speicher}, whose one-dimensional case (which is needed here) was really proved in \cite{Larsen}.  In short: with $a$ $\mathscr{R}$-diagonal normalized so that $\|a\|_2=1$, the spectral radius of $a$ is $\le 1$ and so for $\ll>1$ we may write $(\ll-a) = \ll(1-ra)$ where $r=1/\ll<1$ as in Section \ref{sect PSD 1}.  Then we may expand
\[ (\ll-a)^{-1} = r(1-ra)^{-1} = r\sum_{n=0}^\infty r^n a^n. \]
Hence $\|(\ll-a)^{-1}\| \le r\sum_n r^n \|a^n\|$.  By Corollary 3.2 in \cite{Larsen}, $\|a^n\| \le \sqrt{e}\sqrt{n}\, \|a\|$ since $\|a\|_2=1$, and so we have $(\ll-a)^{-1} \le \sqrt{e}\, r \|a\| \sum_n \sqrt{n}\,r^n = \sqrt{e}\,\ll^{-1} \|a\| \sum_n \sqrt{n} e^{-\ln\ll\, n}$.  It is well-known that the series $\sum_n n^p e^{-tn}$ is bounded above and below by constant multiples of $t^{-p-1}$ for $t,p>0$.  Expanding $\ln\ll$ as a power series in $(\ll-1)$, we then have
\begin{equation} \label{eq bad upper bound} \|(\ll-a)^{-1}\| \le \sqrt{e}\,\|a\| \ll^{-1} (\ln\ll)^{-3/2} \asymp \|a\| (\ll-1)^{-3/2}, \quad \ll\downarrow 1. \end{equation}

\begin{remark} In addition to a non-sharp constant, the estimate of Equation \ref{eq bad upper bound} scales with $\|a\|$ rather than the correct quantity $\sqrt{v(a)}$ from Theorem \ref{main theorem}.  Since both $v(a)$ and $\|a\|$ can be expressed in terms of $aa^*$, which (for $\mathscr{R}$-diagonal $a$) can have arbitrary compactly-supported distribution on $[0,\infty)$, it is easy to see that $\|a\|$ can be arbitrarily large compared to $\sqrt{v(a)}$.  Hence the analytic argument of Section \ref{sect general case} is necessary to get these sharp results. \end{remark}

\begin{remark} In fact, it is true that $\sqrt{v(a)} \le \|a\|$ under the normalization $\|a\|_2=1$; this is not obvious since the two sides scale differently.  Note that $v(a) = \|a\|_4^4 -1 = \|aa^*\|_2^2-1$.  Let $\nu$ be the distribution of $aa^*$; then the supremum of $\supp\nu$ is $\|aa^*\| = \|a\|^2$.  The condition $\|a\|_2 =1$ means that the mean of $aa^\ast$ is $1$, and so $\int_0^{\|a\|^2} t \mu(dt) =1$.  But this means that the measure $u(dt) = t \mu(dt)$ is also a probability measure on $[0,\|a\|^2]$, and thus we have
\[ \|a\|_4^4 = \|aa^*\|_2^2 = \int_0^{\|a\|^2} t^2\,\mu(dt) = \int_0^{\|a\|^2} t\,u(dt) \le \|a\|^2. \]
This shows that the upper-bound of Equation \ref{eq main theorem} is an improvement over the one derived from the strong Haagerup inequality. \end{remark}

\medskip

In fact, the {\em lower bound} of Equation \ref{eq main theorem} in its sharp form can be proved easily from Theorem \ref{thm neg moments}, which can itself be seen as a combinatorial result \`a la Section \ref{sect PSD 1}.  The idea is to use the following simple estimate.  Let $x,y$ be bounded commuting positive semi-definite operators.  Then $x\le\|x\|$ in operator sense, and so $xy \le \|x\|y$.  Applying this with $y = x^k$ for some positive integer $k$, we have $x^{k+1} \le \|x\| x^k$; continuing inductively this yields $x^{k+1} \le \|x\|^k x$, and so applying a state $\phi$ to both sides, $\phi(x^{k+1}) \le \|x\|^k \phi(x)$. 

\medskip

Now, apply this to $x=R_a(\ll)^\ast R_a(\ll)$, so that
\[ \begin{aligned} \|x\| &= \|R_a(\ll)\|^2, \\
\phi(x) &= \phi(R_a(\ll)R_a(\ll)^\ast) = \|(\ll-a)^{-1}\|_2^2 = m_{-2}(\mu_\ll), \\
\phi(x^{k+1}) &= \phi\left( \left((\ll-a)(\ll-a)^{\ast}\right)^{-(k+1)} \right) = m_{-2k-2}(\mu_\ll).
\end{aligned} \]
Thus, we have
\begin{equation} \label{eq norm moment est}
\|(\ll-a)^{-1}\|^{2k} \ge \frac{m_{-2k-2}(\mu_\ll)}{m_{-2}(\mu_\ll)}.
\end{equation}
Applying Equation \ref{eq neg moments}, this ratio is asymptotic to the following:
\[ \frac{m_{-2k-2}(\mu_\ll)}{m_{-2}(\mu_\ll)} \sim \frac{C^{(2)}_k v(a)^k (\ll^2-1)^{-3k-1}}{C^{(2)}_0 v(a)^0 (\ll^2-1)^{-1}} = C^{(2)}_k v(a)^k (\ll^2-1)^{-3k}, \quad \ll\downarrow 1. \]
Taking $2k$th roots, we see from Equation \ref{eq norm moment est} that as $\ll\downarrow 1$, $\|(\ll-a)^{-1}\|$ is bounded below by
\[ \left(C^{(2)}_k\right)^{1/2k} \sqrt{v(a)} (\ll^2-1)^{-3/2} \sim \left(C^{(2)}_k\right)^{1/2k} 2^{-3/2} \sqrt{v(a)}(\ll-1)^{-3/2}. \]
Stirling's formula shows that
\[ \sup_k\left(C^{(2)}_k\right)^{1/2k} = \lim_{k\to\infty}\left(C^{(2)}_k\right)^{1/2k} = \frac{3}{2}\sqrt{3}, \]
and this gives precisely the sharp constant $\sqrt{\frac{27}{32}}$ in Theorem \ref{main theorem}.

\end{document}